\documentclass[12pt]{amsart}
\usepackage{amsmath}
\usepackage{amsxtra}
\usepackage{amscd}
\usepackage{amsthm}
\usepackage{amsfonts}
\usepackage{amssymb}
\usepackage {pstricks}
\usepackage{pstricks,pst-node}

\newtheorem{thm}{Theorem}[section]
\newtheorem{assertion}[thm]{}
\newtheorem{lem}[thm]{Lemma}

\newtheorem{cor}[thm]{Corollary}
\newtheorem{prop}[thm]{Proposition}

\theoremstyle{definition}
\newtheorem{definition}[thm]{Definition}

\theoremstyle{remark}
\newtheorem{rem}[thm]{Remark}

\numberwithin{equation}{section}

\begin{document}
\normalfont
\newcommand{\thmref}[1]{Theorem~\ref{#1}}
\newcommand{\secref}[1]{Section~\ref{#1}}
\newcommand{\lemref}[1]{Lemma~\ref{#1}}
\newcommand{\propref}[1]{Proposition~\ref{#1}}
\newcommand{\corref}[1]{Corollary~\ref{#1}}
\newcommand{\remref}[1]{Remark~\ref{#1}}
\newcommand{\eqnref}[1]{(\ref{#1})}
\newcommand{\exref}[1]{Example~\ref{#1}}

\newcommand{\nc}{\newcommand}

\nc{\on}{\operatorname}

\nc{\Z}{{\mathbb Z}}
\nc{\C}{{\mathbb C}}
\nc{\R}{{\mathbb R}}
\nc{\bbP}{{\mathbb P}}
\nc{\bF}{{\mathbb F}}

\nc{\boldD}{{\mathbb D}}
\nc{\oo}{{\mf O}}
\nc{\N}{{\mathbb N}}
\nc{\bib}{\bibitem}
\nc{\pa}{\partial}
\nc{\F}{{\mf F}}
\nc{\CA}{{\mathcal A}}
\nc{\CE}{{\mathcal E}}
\nc{\CP}{{\mathcal P}}
\nc{\CO}{{\mathcal O}}
\nc{\CK}{{\mathcal K}}
\nc{\Res}{\text{Res}}
\nc{\Ind}{\text{Ind}}
\nc{\Ker}{\text{Ker}}
\nc{\id}{\text{id}}

\nc{\be}{\begin{equation}}
\nc{\ee}{\end{equation}}

\nc{\rarr}{\rightarrow}
\nc{\larr}{\longrightarrow}
\nc{\al}{\alpha}
\nc{\ri}{\rangle}
\nc{\lef}{\langle}

\nc{\W}{{\mc W}}
\nc{\gam}{\ol{\gamma}}
\nc{\Q}{\ol{Q}}
\nc{\q}{\widetilde{Q}}
\nc{\la}{\lambda}
\nc{\ep}{\epsilon}

\nc{\g}{\mf g}
\nc{\h}{\mf h}
\nc{\n}{\mf n}
\nc{\bb}{\mf b}
\nc{\G}{{\mf g}}

\nc{\D}{\mc D}
\nc{\cE}{\mc E}
\nc{\CC}{\mc C}
\nc{\CH}{\mc H}
\nc{\CT}{\mc T}
\nc{\CI}{\mc I}
\nc{\CR}{\mc R}

\nc{\UK}{{\mc U}_{\CA_q}}

\nc{\CS}{\mc S}

\nc{\CB}{\mc B}

\nc{\Li}{{\mc L}}
\nc{\La}{\Lambda}
\nc{\is}{{\mathbf i}}
\nc{\V}{\mf V}
\nc{\bi}{\bibitem}
\nc{\NS}{\mf N}
\nc{\dt}{\mathord{\hbox{${\frac{d}{d t}}$}}}
\nc{\E}{\mc E}
\nc{\ba}{\tilde{\pa}}
\nc{\half}{\frac{1}{2}}

\def\smapdown#1{\big\downarrow\rlap{$\vcenter{\hbox{$\scriptstyle#1$}}$}}

\nc{\mc}{\mathcal}
\nc{\ov}{\overline}
\nc{\mf}{\mathfrak}
\nc{\ol}{\fracline}
\nc{\el}{\ell}
\nc{\etabf}{{\bf \eta}}
\nc{\zetabf}{{\bf
\zeta}}\nc{\x}{{\bf x}}
\nc{\xibf}{{\bf \xi}} \nc{\y}{{\bf y}}
\nc{\WW}{\mc W}
\nc{\SW}{\mc S \mc W}
\nc{\sd}{\mc S \mc D}
\nc{\hsd}{\widehat{\mc S\mc D}}
\nc{\parth}{\partial_{\theta}}
\nc{\cwo}{\C[w]^{(1)}}
\nc{\cwe}{\C[w]^{(0)}} \nc{\wt}{\widetilde}
\nc{\gl}{\mf gl}
\nc{\K}{\mf k}

\newcommand{\U}{{\rm{U}}}
\newcommand{\End}{{\rm{End}}}
\newcommand{\Hom}{{\rm{Hom}}}
\newcommand{\Lie}{{\rm{Lie}}}
\newcommand{\Uq}{{{\rm U}_q}}
\newcommand{\GL}{{\rm{GL}}}
\newcommand{\Sym}{{\rm{Sym}}}
\newcommand{\rk}{{\rm{rk}}}
\newcommand{\tr}{{\rm{tr}}}
\newcommand{\Rea}{{\rm{Re}}}
\newcommand{\rank}{{\rm{rank}}}
\newcommand{\im}{{\rm{Im}}}

\advance\headheight by 2pt

\nc{\fb}{{\mathfrak b}} \nc{\fg}{{\mathfrak g}}

\nc{\fh}{{\mathfrak h}}  \nc{\fk}{{\mathfrak k}}

\nc{\fl}{{\mathfrak l}} \nc{\fn}{{\mathfrak n}}

\nc{\fp}{{\mathfrak p}} \nc{\fu}{u}


\nc{\fS}{{\Sym}}

\nc{\fsl}{{\mathfrak {sl}}} \nc{\fsp}{{\mathfrak {sp}}}
\nc{\fso}{{\mathfrak {so}}} \nc{\fgl}{{\mathfrak {gl}}}

\nc{\A}{\mc A} \nc{\cF}{{\mathcal F}}

\nc{\cA}{{\mathcal A}} \nc{\cP}{{\mathcal P}} \nc{\cC}{{\mathcal C}}
\nc{\cU}{{\mathcal U}} \nc{\cB}{{\mathcal B}}

\def\lr{{\longrightarrow}}
\def\inv{{^{-1}}}

\def\xl{{\hbox{\lower 2pt\hbox{$\scriptstyle \mathfrak L$}}}}

\nc{\bX}{{\mathbf X}} \nc{\bx}{{\mathbf x}} \nc{\bd}{{\mathbf d}}
\nc{\bdim}{{\mathbf dim}} \nc{\bm}{{\mathbf m}}

\title[Temperley-Lieb]{A Temperley-Lieb analogue for the BMW algebra}

\author{G.I. Lehrer and R.B. Zhang}
\address{School of Mathematics and Statistics,
University of Sydney, N.S.W. 2006, Australia}
\email{gusl@maths.usyd.edu.au, rzhang@maths.usyd.edu.au}
\date {23rd November, 2007}

\dedicatory{To Toshiaki Shoji on his 60${}^{\text {\it \,th}}$
birthday}

\begin{abstract} The Temperley-Lieb algebra may be thought of
as a quotient of the Hecke algebra of type $A$, acting on tensor
space as the commutant of the usual action of quantum $\fsl_2$ on
$(\C(q)^2)^n$. We define and study a quotient of the
Birman-Wenzl-Murakami algebra, which plays an analogous role for the
$3$-dimensional representation of quantum 
$\fsl_2$. In the course of the discussion
we prove some general results about the radical of a cellular algebra,
which may be of independent interest.
\end{abstract}
\maketitle
\tableofcontents
\section{Introduction} Let $\fg$ be a finite dimensional
simple complex Lie algebra, $\cU(\fg)$ its universal enveloping
algebra, and $\cU_q=\cU_q(\fg)$ its Drinfeld-Jimbo quantisation, the
latter being an algebra over the function field
$\CK:=\C(q^{\frac{1}{2}})$, $q$ an indeterminate. As explained in
\cite[\S 6]{LZ}, the finite dimensional $\fg$-modules correspond
bijectively to the ``type $(1,1,\dots,1)$ modules'' of $\cU_q$, with
corresponding modules having the same character (which is an element
of the weight lattice of $\fg$).

Let $V$ be an irreducible finite dimensional $\fg$-module, and $V_q$
its $q$-analogue (the corresponding $\cU_q$-module). It is known
that there is an action of the $r$-string braid group $B_r$ on the
tensor space $V_q^{\otimes r}$ which commutes with the action of
$\cU_q$, and in \cite{LZ} a sufficient condition was given in order
that $B_r$ span $\End_{\cU_q}V_q^{\otimes r}$. This condition, that
$V$ be ``strongly multiplicity free'' (see \cite[\S 3]{LZ}) was
shown to be satisfied when $V$ is any irreducible module for
$\fsl_2$.

When $V=V(1)$, the (natural) two-dimensional $\fsl_2$-module, it is
known that the two braid generators satisfy a quadratic relation,
and together with the quantum analogue of the relation which
expresses the vanishing of alternating tensors of rank $\geq 3$
these relations give a presentation of the Temperley-Lieb algebra
(see \S 2 below and \cite{GL03}). In this work we study the algebra
which occurs when we start with the three-dimensional
$\fsl_2$-module $V(2)$. It follows from our earlier
work that this algebra is a
quotient of the BMW (Birman-Wenzl-Murakami) algebra, and we give a
presentation for a quotient of the latter, whose semisimple quotient
specialises generically to the endomorphism algebra
of tensor space. One of the major differences between our case
and the classical Temperley-Lieb case is that neither the BMW
algebra we start with nor its Brauer specialisation at $q=1$ is
semisimple. 

This work makes extensive use of the cellular structure of the BMW algebra
and its ``classical'' specialisation, the Brauer algebra. We use 
specialisation arguments to relate the quantum and classical
($q=1$) situations. Because of this, and also because we have in mind
applications of this work to cases where the modules concerned may
not be semisimple, we shall work in integral lattices for the
modules we encounter, and with integral forms of the endomorphism algebras.
Such constructions are closely related to the
``Lusztig form'' of the irreducible $\cU_q$-modules.

In \S\ref{radical} we prove some general results concerning the radical of 
a cellular algebra. These characterise it quite explicitly, and
give a general criterion for an ideal to contain the radical. These 
results may have some interest, independently of the rest of this work. 

\section{Dimensions} Let $\fg=\fsl_2$ and write
$V(d)$ for the $(d+1)$-dimensional irreducible representation of
$\fg$ on homogeneous polynomials of degree $d$ in two variables, say
$x$ and $y$. The standard generators $e,f,h$ of $\fsl_2$ act as
$x\frac{\partial}{\partial y}, y\frac{\partial}{\partial x},
x\frac{\partial}{\partial x}-y\frac{\partial}{\partial y}$
respectively.

\subsection{Dimension of the endomorphism algebra}
We start by giving a (well known) recursive formula for
$\dim_\C\End_{\fg}V(d)^{\otimes r}=
\dim_{\C(z)}\End_{\cU_q}V(d)_q^{\otimes r}$. From the classical
Clebsch-Gordan formula, we have for $n\geq m$,
\begin{equation}\label{cg}
V(n)\otimes V(m)\cong V(n+m)\oplus V(n+m-2)\oplus\dots\oplus V(n-m).
\end{equation}

Define the coefficients $m_n^r(i)$ by
\begin{equation}\label{defmnri}
V(n)^{\otimes r}\cong \bigoplus_{i\equiv rn(\text{mod }2)}m_n^r(i)V(i),
\end{equation}
and the dimension
\begin{equation}\label{defdnr}
d(n,r):=\dim\End_\fg(V(n)^{\otimes r}).
\end{equation}

Then $d(n,r)=\sum_{i\equiv rn(\text{mod }2)} m_n^r(i)^2$. Since all
the modules $V(n)$ are self dual, it also follows from Schur's lemma
that $(V(n)^{\otimes r},V(n)^{\otimes r})_\fg=(V(n)^{\otimes
2r},V(0))_\fg$, where $(\;,\;)_\fg$ denotes multiplicity. Thus
$$
d(n,r)=\sum_{i\equiv rn(\text{mod }2)}m_n^r(i)^2=m_n^{2r}(0).
$$

Using (\ref{cg}), it is a straightforward combinatorial exercise to
prove the following recursive formula. Let $x$ be an indeterminate
and write $[n]_x=\frac{x^{n}-x^{-n}}{x-x^{-1}}\in\Z[x^{\pm 1}]$ for
the ``$x$-analogue'' of $n\in\Z$. Then define the integers
$a_n^r(k)$ by $[n+1]_x^r=\sum_ka_n^r(k)x^k$, where the sum is over
$k$ such that $-nr\leq k\leq nr$ and $k\equiv nr\text{mod }2$.
Finally, define the integers $b_n^r(k)$ ($-nr\leq k\leq nr$) by
downward recursion on $k$ as follows.

Set $b_n^r(nr)=a_n^r(nr)=1$, and then $\sum_{i\geq
k}^{nr}b_n^r(i)=a_n^r(k)$. Equivalently,
$b_n^r(k)=a_n^r(k)-a_n^r(k+1)$

\begin{prop}\label{dimend}
We have, for all $n,r$ and $k$,
$$m_n^r(k)=b_n^r(k).$$
\end{prop}

\subsection{The case $n=1$} In this case one verifies easily
that
$$
a_1^r(k)=\begin{cases} \binom{r}{\frac{k+r}{2}}\text{ if } k\equiv
r\text{ (mod }2)\\0\text{ otherwise}\\ \end{cases}.
$$

It is then straightforward to compute that
$$
m_1^r(k)=\binom{r}{\frac{r+k}{2}}\frac{2(k+1)}{r+k+2}.
$$

In particular,
\begin{equation}\label{dimtl}
d(1,r)=\dim\End_{\fsl_2}V(1)^{\otimes r}=m_1^{2r}(0)
={\frac{1}{r+1}}{\binom{2r}{r}}.
\end{equation}

\subsection{The case $n=2$} With the above notation,
we have
$$
a_2^r(2\ell)=\sum_{k\geq \frac{\ell+r}{2}}^r\frac{r!}
{(2k-(\ell+r))!(\ell+r-k)!(r-k)!}
$$

In this case, we have from Proposition \ref{dimend}, that
$m_2^r(2\ell)=a_2^r(2\ell)-a_2^r(2\ell+2)$. This relation easily
yields the following formula for $d(2,r)$.
\begin{equation}\label{d2r}
\begin{aligned}
d(2,r)&=\dim\End_{\fsl_2}V(2)^{\otimes r}=m_2^{2r}(0)\\
&=\binom{2r}{r}+\sum_{p=0}^{r-1}\binom{2r}{2p}\binom{2p}{p}
\frac{3p-2r+1}{p+1}.
\end{aligned}
\end{equation}
Thus for $r=1,2,3,4,5$ the respective dimensions are $1,3,15,91,$
and $603$.

\section{Some generators and relations for
$\End_{\cU_q}(V(n)_q^{\otimes r})$}

In this section, we review the results of \cite{LZ} which pertain to
the structure of the endomorphism algebras we wish to study.

\subsection{The general case}\label{gen}

Recall that with $\fg$ and $\cU_q$ as above, given any
$\cU_q$-module $V_q$, there is an operator $\check R\in
\End_{\cU_q}(V_q\otimes V_q)$, known as an ``$R$-matrix'' (see
\cite[\S 6.2]{LZ}). Denote by $R_i$ the element $\id_{V_q}^{\otimes
i-1}\otimes \check R\otimes \id_{V_q}^{\otimes r-i-1}$ of
$\End_{\cU_q}V_q^{\otimes r}$ ($i=1,\dots,r-1$). It is well known
that the $R_i$ satisfy the braid relations:
\begin{equation}\label{br}
\begin{aligned}
R_iR_j&=R_jR_i\text{ if }|i-j|\geq 2\\
R_iR_{i+1}R_i&=R_{i+1}R_iR_{i+1} \text{ for }1\leq i\leq r-1.\\
\end{aligned}
\end{equation}

Moreover if $V_q$ is strongly multiplicity free (for the definition
see \cite[\S 7]{LZ}), it follows from \cite[Theorem 7.5]{LZ} that
the endomorphisms $R_i$ generate $\End_{\cU_q}(V_q^{\otimes r})$.
Assume henceforth that $V_q$ is strongly multiplicity free. The
following facts may be found in \cite[\S\S 3,7]{LZ}.

Firstly, $V_q=L_{\lambda_0}$, the unique irreducible module for
$\cU_q$ with highest weight $\lambda_0$, and $V_q\otimes V_q$ is
multiplicity free as $\cU_q$-module. Write
$$
V_q\otimes V_q\cong\oplus_{\mu\in\CP(\lambda_0)}L_\mu,
$$
where $\CP_{\lambda_0}$ is the relevant set of dominant weights of
$\fg$, and $L_\mu$ is the irreducible $\cU_q$- module with highest
weight $\mu$. Let $P(\mu)$ be the projection $:V_q\otimes V_q\lr
L_\mu$. These projections clearly span $\End_{\cU_q}(V_q\otimes
V_q)$, and we have (see \cite[(6.10)]{LZ})
\begin{equation}\label{eigenR}
\check R=\sum_{\mu\in\CP_{\lambda_0}}\varepsilon(\mu)
q^{\frac{1}{2}(\chi_\mu(C)-2\chi_{\lambda_0}(C))}P(\mu),
\end{equation}
where $C\in\cU(\fg)$ is the classical quadratic Casimir element,
$\chi_\lambda(C)$ is the scalar through which $C$ acts on the
(classical) irreducible $\cU(\fg)$-module with highest weight
$\lambda$, and $\varepsilon(\mu)$ is the sign occurring in the
action of the interchange $s$ on the classical limit $V\otimes V$ of
$V_q\otimes V_q$ as $q\to 1$.

\subsection{Relations for the case $\fg=\fsl_2$}

It was proved in \cite{LZ} that all irreducible modules for
$\cU_q(\fsl_2)$ are strongly multiplicity free. In this subsection,
we make explicit the relations above when $\fg=\fsl_2$ and
$V_q=V(n)_q$. These statements are all well known. As above, we
think of $V(n)$ as the space $\C[x,y]_n$ of homogeneous polynomials
of degree $n$. This has highest weight $n$, with
$h=x\frac{\partial}{\partial x}-y\frac{\partial}{\partial y}$ acting
on $x^n$ as highest weight vector. We have seen that
\begin{equation}\label{square}
V(n)\otimes V(n)=\oplus_{\ell=0}^{n}V(2\ell).
\end{equation}

It is easy to compute the highest weight vectors in the summands of
(\ref{square}), which leads to
\begin{assertion}\label{varepsilon}
The endomorphism $s:v\otimes w\mapsto w\otimes v$ of $V(n)\otimes
V(n)$ acts on $V(2\ell)$ as $(-1)^{n+\ell}$. That is, in the
notation of \S\ref{gen}, $\varepsilon(2\ell)=(-1)^{d+\ell}$.
\end{assertion}

Next, if $\theta= x\frac{\partial}{\partial
x}+y\frac{\partial}{\partial y}$ is the Euler form, the Casimir $C$
is given by $C=\theta+\frac{1}{2}\theta^2$. It follows that
\begin{assertion}\label{chimu}
The Casimir $C$ acts on $V(n)$ as multiplication by
$\chi_n(C)=\frac{1}{2}n(n+2)$.
\end{assertion}

Applying the statements in \S\ref{gen} here, we obtain

\begin{prop}\label{projections}
Let $V$ be the irreducible $\fsl_2$ module $V(n)$, with $V(n)_q$ its
quantum analogue. Then $V(n)_q\otimes
V(n)_q=\oplus_{\ell=0}^{n}V(2\ell)_q$, and if $\check R$ is the
$R$-matrix acting on $V(n)_q\otimes V(n)_q$, then
\begin{equation}\label{r-as-proj}
\check R=\sum_{\ell=0}^n(-1)^{n+\ell}
q^{\frac{1}{2}(\ell(2\ell+2)-n(n+2))}P(2\ell),
\end{equation}
where $P(2\ell)$ is the projection to the component $V(2\ell)_q$.
\end{prop}

Now let $E_q(n,r):=\End_{\cU_q}(V(n)_q^{\otimes r})$, and let
$R_i\in E_q(n,r)$ be the endomorphism defined above
($i=1,\dots,r-1$). We have seen that the $R_i$ generate $E_q(n,r)$,
and that they satisfy the relations (\ref{br}). From the relation
(\ref{r-as-proj}), we deduce that for all $i$,
\begin{equation}\label{poly-r}
\prod_{\ell=0}^n\left(R_i-(-1)^{n+\ell}
q^{\frac{1}{2}(\ell(2\ell+2)-n(n+2))}\right)=0.
\end{equation}

Writing $T_i:=(-1)^nq^{\frac{1}{2}n(n+2)}R_i$, the above relation
simplifies to
\begin{equation}\label{poly-T}
\prod_{\ell=0}^n\left(T_i-(-1)^{\ell} q^{\ell(\ell+1)}\right)=0.
\end{equation}

Now the relations (\ref{br}) and (\ref{poly-r}) do not provide a
presentation for $E_q(n,r)$, and it is one of our objectives to
determine further relations among the $R_i$. These will suffice to
present $E_q(n,r)$ as an associative algebra generated by the $R_i$
only in some special cases.

\subsection{Basic facts about $\cU_q(\fsl_2)$}\label{basic}
It will be convenient to establish notation for the discussion below
by recalling the following basic facts. The algebra
$\cU_q:=\cU_q(\fsl_2)$ has generators $e,f$ and $k^{\pm 1}$, with
relations $kek\inv=q^2e$, $kfk\inv=q^{-2}f$,  and
$ef-fe=\frac{k-k\inv}{q-q\inv}$. The weight lattice $P$ is
identified with $\Z$, and for $\lambda,\mu\in P$,
$(\lambda,\mu)=\frac{1}{2}\lambda\mu$. In general, if $M$ is a
$\cU_q$-module with weights $\lambda_1,\dots,\lambda_d$ ($d=\dim
M$), then the quantum dimension of (i.e. quantum trace of the
identity on) $M$ is $\dim_q M=\sum_{i=1}^dq^{-(2\rho,\lambda_i)}$,
where $2\rho$ is the sum of the positive roots, in this case $2$.
Hence $\dim_qV(n)_q=q^n+q^{n-2}+\dots+q^{-n}=[n+1]_q$. The
comultiplication is given by $\Delta(e)=e\otimes k+1\otimes e$,
$\Delta(f) =f\otimes 1+k^\inv\otimes f$, $\Delta(k)=k\otimes k$.

\subsection{The case $n=1$: the Temperley-Lieb algebra}\label{sectionhecke}

In this case the relation (\ref{poly-r}) reads
$$
(R_i+q^{\frac{-3}{2}})(R_i-q^{\frac{1}{2}})=0.
$$
Renormalising by setting $T_i=q^{\frac{1}{2}}R_i$
($i=1,2,\dots,r-1$) we obtain
\begin{equation}\label{hecke}
(T_i+q^{-1})(T_i-q)=0.
\end{equation}

Now it is well known that the associative $\C(q)$-algebra with
generators $T_1,\dots,T_{r-1}$ and relations (\ref{br}) and
(\ref{poly-r}) is the Hecke algebra $H_r(q)$ of type $A_{r-1}$ with
parameter $q$. The algebra $H_r(q)$ has a $\C(q)$-basis $\{T_w\mid
w\in\Sym_r\}$, and we may therefore speak of the action of
$T_w(=q^{\frac{\ell(w)}{2}}R_w)$ (where $\ell(w)$ is the usual
length function in $\Sym_r$) on $V_q^{\otimes r}$.

Evidently we have a surjection $\phi:H_r(q)\lr E_q(1,r)$, and we
shall determine $\ker \phi$. The next statement is just the quantum
analogue of the fact that there are no non-zero alternating tensors
in $V^{\otimes 3}$ if $V$ is $2$-dimensional.

\begin{lem}\label{tl-alternator}
Let $V=V(1)_q$ be the two-dimensional irreducible $\cU_q(\fsl_2)$-
module. In the notation above, let $E(\varepsilon)=
\sum_{w\in\Sym_3}(-q)^{-\ell(w)}T_w\in H_3(q)$. Then
$E(\varepsilon)$ acts as zero on $V^{\otimes 3}$.
\end{lem}
\begin{proof}
Let $v_1\in V$ have weight $1$, and take $v_2=fv_1$ as the
complementary basis element, which has weight $-1$. We have
$V\otimes V\cong L_0\oplus L_1$, where $L_0$ is the trivial $\cU_q$
module, and $L_1$ is the irreducible $\cU_q$-module of dimension
$3$. By computing the action of $\Delta(e)$ and $\Delta(f)$, one
sees easily that $L_0$ is spanned by $qv_1\otimes v_2-v_2\otimes
v_1$, and $L_1$ has basis $v_1\otimes v_1$, $v_2\otimes v_2$ and
$v_2\otimes v_1+q^\inv v_1\otimes v_2$.

Now $T=q^{\frac{1}{2}}\check R$ acts on $L_0$ as $-q^{-1}$ and on
$L_1$ as $q$. If $P(0),P(1)$ are the projections of $V\otimes V$
onto $L_0,L_1$ respectively, this implies that
$P(0)=-\frac{1}{q+q^{-1}}(T-q)$, and
$P(1)=\frac{1}{q+q^{-1}}(T+q^{-1})$. As above, write $P_i(j)$ for
the projection of $V^{\otimes r}$ obtained by applying $P(j)$ to the
$(i,i+1)$ factors of $V^{\otimes r}$ ($i=1,\dots,r-1;\;\;j=0,1)$.
Then $P_i(1)=-\frac{1}{q+q^{-1}}(T_i+q^{-1})$, etc.

Next observe that since $(T_i+q^{-1})E(\varepsilon)
=E(\varepsilon)(T_i+q^{-1})=0$ for $i=1,2$, we have
$P_i(1)E(\varepsilon)=E(\varepsilon)P_i(1)=0$ for $i=1,2$. Since
$P_i(0)+P_i(1)=\id_{V^{\otimes 3}}$, it follows that
$$
\begin{aligned}
E(\varepsilon)V^{\otimes 3}&\subseteq P_1(0)V^{\otimes 3}
\cap P_2(0)V^{\otimes 3}\\
&=L_0\otimes V\cap V\otimes L_0.
\end{aligned}
$$
Now $L_0\otimes V$ and $V\otimes L_0$ are two irreducible submodules
of $V^{\otimes 3}$. Hence they either coincide or have zero
intersection. But $L_0\otimes V$ has basis $\{(qv_1\otimes
v_2-v_2\otimes v_1)\otimes v_i\mid i=1,2\}$. Hence $qv_1\otimes
v_1\otimes v_2-v_1\otimes v_2\otimes v_1$ is in $V\otimes L_0$, but
not in $L_0\otimes V$.

It follows that $E(\varepsilon)V^{\otimes 3}=0$.
\end{proof}

This enables us to prove

\begin{thm}
Let $V$ be the two-dimensional irreducible representation of $\cU_q(\fsl_2)$.
For each integer $r\geq 2$, let
$E_q(1,r)=\End_{\cU_q(\fsl_2)}(V^{\otimes r})$. Then $E_q(1,r)$ is
isomorphic to the Temperley-Lieb algebra $TL_r(q)$ (cf. \cite[p.
144]{GL04}).
\end{thm}
\begin{proof}
We have seen above that $E_q(1,r)$ is generated as $\CK$-algebra by
the endomorphisms $T_1,\dots,T_{r-1}$. By (\ref{hecke}), they
generate a quotient of the Hecke algebra $H_r(q)$ and by Lemma
\ref{tl-alternator}, that quotient is actually a quotient of
$H_r(q)/I$, where $I$ is the ideal generated by the element
$E(\varepsilon)$ defined above. But up to the automorphism
$T_i\mapsto T_i'=-T_i+q-q^{-1}$, this is precisely the idempotent
$E_1$ of \cite[(4.9)]{GL04}. It follows (see, e.g.
\cite[(4.17]{GL04}) that the quotient $H_r(q)/I$ is isomorphic to
$TL_r(q)$.

But this latter algebra is well known (cf. \cite{GL96} or
\cite{GL03}) to have dimension $\frac{1}{r+1}\binom {2r}{r}$, which
by (\ref{dimtl}) above is the dimension of $E_q(1,r)$. The theorem
follows.
\end{proof}

It follows that in this case, the endomorphism algebra has a well
understood cellular structure (see \cite{GL96}).

\section{The case $n=2$: action of the BMW algebra}

In this section we take $V$ to be $V_q(2)$, the three-dimensional
irreducible module for $\cU_q(\fsl_2)$. In accord with the notation
of the last section, we write
$E_q(2,r):=\End_{\cU_q(\fsl_2)}V_q(2)^{\otimes r}$.

\subsection{The setup, and some relations} In this situation,
$V\otimes V\cong V_0\oplus V_1\oplus V_2$, where $V_0$ is the
trivial module, and $V_1,V_2$ are respectively the three and five
dimensional irreducible modules. As above, we therefore have
operators $R_i$, $P_i(j)$ ($i=1,\dots,r-1;\;j=0,1,2$) on $V^{\otimes
r}$, where $P_i(j)$ is the projection $V\otimes V\to V_j$, applied
to the $(i,i+1)$ factors of $V^{\otimes r}$, appropriately tensored
with the identity endomorphism of $V$.

The $R_i$ here satisfy the braid relations (as they always do), and
the cubic relation
\begin{equation}\label{cubicr}
(R_i-q^{-4})(R_i+q^{-2})(R_i-q^2)=0.
\end{equation}

Now if $L$ is any strongly multiplicity free $\cU_q$-module such
that the trivial module $L_0$ is a summand of $L\otimes L$, and
$f\in\End_{\cU_q}(L\otimes L)$, then writing $P_i(0)$ for the
projection $L\otimes L\to L_0$, applied to the $(i,i+1)$ components
of $L^{\otimes r}$, we have
\begin{equation}\label{proj-gen}
P_i(0)f_{i\pm 1}P_i(0)= \frac{1}{(\dim_q(L))^2}\tau_{q,L\otimes
L}(f)P_i(0),
\end{equation}
where $\tau_{q,M}$ denotes the quantum trace of an endomorphism of
the $\cU_q$-module $M$, and $f_i$ is $f$ applied to the $(i,i+1)$
components on $L^{\otimes r}$.

To apply (\ref{proj-gen}) to the case when $f=\check R$, we shall
use
\begin{equation}\label{qtracer}
\begin{aligned}
\tau_{q,V(n)_q^{\otimes 2}}(\check R)&=q^{\frac{1}{2}n(n+2)}[n+1]_q\\
&=q^{\frac{1}{2}n(n+2)}\frac{q^{n+1}-q^{-(n+1)}}{q-q\inv}.\\
\end{aligned}
\end{equation}
This may be proved in several different ways, including the use of
the explicit expression given in Proposition \ref{projections} for
$\check R$.

Applying (\ref{proj-gen}) to the cases $f=\check R$ and $f=P(0)$ in
turn, we obtain for our case $L=V_q(2)$,
\begin{equation}\label{braidpr}
P_i(0)R_{i\pm 1}P_i(0)=q^{4}[3]_q\inv P_i(0)\text{ for }i=1,\dots,r.
\end{equation}
and
\begin{equation}\label{braidpp}
P_i(0)P_{i\pm 1}(j)P_i(0)=[2j+1]_q[3]_q^{-2} P_i(0) \text{ for
}i=1,\dots,r\;\;\text{ and }j=0,1,2.
\end{equation}

In the equations (\ref{braidpr}),(\ref{braidpp}), the applicable
range of values for $i$ is understood to be such that $P_{k(i)}(j)$
and $R_{k(i)}$ makes sense for the relevant functions $k(i)$ of $i$.

Since $\check R$ acts on $V_0,V_1$ and $V_2$ respectively as
$q^{-4},-q^{-2}$ and $q^2$, we also have
\begin{equation}\label{p0}
P_i(0)=\frac{q^8(R_i+q^{-2})(R_i-q^2)}{(1+q^2)(1-q^6)}.
\end{equation}
\subsection{The BMW algebra}\label{bmw-basic}

We recall some basic facts concerning the BMW algebra,
suitably adapted to our context. Let
$\CK=\C(q^{\frac{1}{2}})$ as above, and let $\CA$ be the ring
$\C[y^{\pm 1},z]$, where $y,z$ are indeterminates.

The BMW algebra $BMW_r(y,z)$ over $\CA$ is the associative
$\CA$-algebra with generators $g_1^{\pm 1},\dots,g_{r-1}^{\pm
1}$ and $e_1,\dots,e_{r-1}$, subject to the following relations:

The braid relations for the $g_i$:
\begin{equation}\label{braidgi}
\begin{aligned}
g_ig_j&=g_jg_i\text{ if }|i-j|\geq 2\\
g_ig_{i+1}g_i&=g_{i+1}g_ig_{i+1} \text{ for }1\leq i\leq r-1;\\
\end{aligned}
\end{equation}
The Kauffman skein relations:
\begin{equation}\label{kauffman}
g_i-g_i\inv=z(1-e_i)\text { for all }i;
\end{equation}
The de-looping relations:
\begin{equation}\label{delooping}
\begin{aligned}
&g_ie_i=e_ig_i=ye_i;\\\
&e_ig_{i-1}^{\pm 1}e_i=y^{\mp 1}e_i;\\
&e_ig_{i+1}^{\pm 1}e_i=y^{\mp 1}e_i.\\
\end{aligned}
\end{equation}

The next four relations are easy consequences of the previous three.
\begin{eqnarray}
&&e_ie_{i\pm 1}e_i=e_i; \label{bmwtl}\\
&&(g_i-y)(g_i^2-zg_i-1)=0; \label{cubic}\\
&&ze_i^2=(z+y\inv -y)e_i,\quad  
\label{esquared}\\
&&-yze_i=g_i^2-zg_i-1. \label{equadg}
\end{eqnarray}

It is easy to show that $BMW_r(y,z)$ may be defined using the
relations (\ref{braidgi}), (\ref{delooping}), (\ref{cubic}) and
(\ref{equadg}) instead of (\ref{braidgi}), (\ref{kauffman}) and
(\ref{delooping}), i.e. that (\ref{kauffman}) is a consequence of
(\ref{cubic}) and (\ref{equadg}).
 
We shall require a particular specialisation of $BMW_r(y,z)$
to a subring $\CA_q$ of $\CK$, which is defined as follows.
Let $\CS$ be the multiplicative subset of 
$\C[q, q^{-1}]$
generated by $[2]_q$, $[3]_q$ and $[3]_q-1$. Let $\CA_q:=\C[q,
q^{-1}]_\CS:=\C[q,q\inv, [2]_q\inv,[3]_q\inv, (q^2+q^{-2})\inv]$
be the localisation of $\C[q, q^{-1}]$ at $\CS$. 

Now let $\psi:\C[y^{\pm 1},z]\lr\CA_q$ be the homomorphism 
defined by $y\mapsto q^{-4}$, $z\mapsto q^2-q^{-2}$. Then 
$\psi$ makes $\CA_q$ into an $\CA$-module, and the 
specialisation $BMW_r(q):=\CA_q\otimes_{\CA} BMW_r(y,z)$
is the $\CA_q$-algebra with generators which we denote, by abuse of
notation, $g_i^{\pm 1},e_i$ ($i=1,\dots,r-1$) and relations
(\ref{braidgiq}) below, with the relations (\ref{extrarel}) being 
consequences of (\ref{braidgiq}). 

\begin{equation}\label{braidgiq}
\begin{aligned}
g_ig_j&=g_jg_i\text{ if }|i-j|\geq 2\\
g_ig_{i+1}g_i&=g_{i+1}g_ig_{i+1} \text{ for }1\leq i\leq r-1\\
g_i-g_i\inv&=(q^2-q^{-2})(1-e_i)\text { for all }i\\
g_ie_i&=e_ig_i=q^{-4}e_i\\
e_ig_{i-1}^{\pm 1}e_i&=q^{\pm 4}e_i\\
e_ig_{i+1}^{\pm 1}e_i&=q^{\pm 4}e_i.\\
\end{aligned}
\end{equation}

\begin{equation}\label{extrarel}
\begin{aligned}
e_ie_{i\pm 1}e_i&=e_i  \\
(g_i-q^2)(g_i+q^{-2})&=-q^{-4}(q^2-q^{-2})e_i\\
(g_i-q^{-4})(g_i-q^2)(g_i+q^{-2})&=0\\
e_i^2&=(q^2+1+q^{-2})e_i.\\
\end{aligned}
\end{equation}
%


We shall be concerned with the following two specialisations
of $BMW_r(q)$.

\begin{definition}\label{spec}  Let $\phi_q:\CA_q\lr \CK=\C(q^{\frac{1}{2}})$ be the
inclusion map, and let $\phi_1:BMW_r(q)\lr\C$ be the $\C$-algebra homomorphism
defined by $q\mapsto 1$. Define the specialisations
$BMW_r(\CK):=\CK\otimes_{\phi_q}BMW_r(q)$, and 
$BMW_r(1):=\C\otimes_{\phi_1}BMW_r(q)$.
\end{definition}

The next statement is straightforward.

\begin{lem}\label{bmwspec}
\begin{enumerate}
\item The algebra $BMW_r(q)$ may be regarded as an
$\CA_q$-lattice in the $\CK$-algebra $BMW_r(\CK)$.
\item  The specialisation $BMW_r(1)$ 
is isomorphic to the Brauer algebra $B_r(3)$ over
$\C$.
\item Let $\CI$ be the two-sided ideal of $BMW_r(q)$
generated by $e_1,\dots,e_{r-1}$. There is a surjection
$BMW_r(q)\to H_r(q^2)$ of $\CA_q$-algebras with kernel $\CI$,
where $H_r(q^2)$ is the Hecke algebra discussed above (\S \ref{sectionhecke}).
\end{enumerate}
\end{lem}
\begin{proof}
Note for the first two statements, that 
by \cite[Theorem 3.11 and its proof]{X},
$BMW_r(y,z)$ has a basis of ``$r$-tangles'' $\{T_d\}$, 
where $d$ runs over the Brauer $r$-diagrams, which form a 
basis of $B_r(\delta)$ over any ring. The same thing applies
to $BMW_r(y,z)$; thus $BMW_r(q)$ may be thought of as the 
subring of $BMW_r(\CK)$ consisting of the $\CA_q$-linear 
combinations of the $T_d$, while 
$B_r(3)$ is realised as the set of $\C$-linear combinations 
of the diagrams $d$.
\end{proof}
Note that in view of the third relation in
(\ref{extrarel}), the element 
$\frac{e_i}{[3]_q}$ of $BMW_r(q)$ is an idempotent. 
Moreover it follows from (\ref{equadg})
or (\ref{extrarel}) that
\be\label{ei}
\frac{e_i}{[3]_q}=\frac{(g_i-q^2)(g_i+q^{-2})}{(q^{-4}-q^2)(q^{-4}+q^{-2})}.
\ee

Taking into account the cubic relation (\ref{cubic}),
or its specialisation in (\ref{extrarel}), we also have
the idempotents $d_i$ and $c_i$ in $BMW_r(\CK)$, where
\begin{equation}\label{otheridemp}
\begin{aligned}
d_i&=\frac{(g_i-q^2)(g_i-q^{-4})}{(q^{-2}+q^2)(q^{-2}+q^{-4})}\\
c_i&=\frac{(g_i+q^{-2})(g_i-q^{-4})}{(q^2+q^{-2})(q^2-q^{-4})}.\\
\end{aligned}
\end{equation}

\begin{lem}\label{integralidempotents}
If $BMW_r(q)$ is thought of as an $\CA_q$-submodule of 
$BMW_r(\CK)$ as in Lemma \ref{bmwspec}(1),
then the idempotents $e_i[3]_q\inv$, $d_i$
and $c_i$ all lie in $BMW_r(q)$. 
\end{lem}
\begin{proof} 
It is evident that $e_i[3]_q\inv\in BMW_r(q)$.
Since ${(q^{-2}+q^2)(q^{-2}+q^{-4})}$ 
is invertible in 
$\CA_q$, clearly $d_i\in BMW_r(q)$. But it is easily verified that
$e_i[3]_q\inv + d_i + c_i=1$, whence the result.
\end{proof}

The relevance of the above for the study of endomorphisms is
evident from the next result.

\begin{thm}\label{surjbmwend}
With the above notation, there is a surjection $\eta_q$ from the
algebra $BMW_r(\CK)\to E_q(2,r)$ which takes $e_i$ to
$[3]_qP_i(0)$ and $g_i$ to $R_i$.
\end{thm}
\begin{proof}
In view of the above discussion, it remains only to show that the
endomorphisms $\eta_q(g_i)=R_i$ and $\eta_q(e_i)=[3]_qP_i(0)$
satisfy the relations (\ref{braidgi}), (\ref{delooping}),
(\ref{cubic}) and (\ref{equadg}) for the appropriate $y$ and $z$.
Now the braid relations (\ref{braidgi}) are always satisfied by the
$R_i$; further, (\ref{cubic}) with the $R_i$ replacing the $g_i$ is
just (\ref{cubicr}). Now a simple calculation shows that in our
specialisation, $z\inv(z+y\inv-y)=[3]_q:=\delta$. It follows that
(\ref{equadg}) may be written
\[
\begin{aligned}
\delta\inv e_i=&\frac{1}{\delta yz}(g_i-q^2)(g_i+q^{-2})\\
=&\frac{-1}{[3]_q q^{-4}(q^2-q^{-2})}(g_i-q^2)(g_i+q^{-2})\\
=&\frac{-(q-q\inv)q^4}{(q^3-q^{-3})(q^2-q^{-2})}(g_i-q^2)(g_i+q^{-2})\\
=&\frac{q^8}{(1+q^{2})(1-q^{6})}(g_i-q^2)(g_i+q^{-2}).\\
\end{aligned}
\]
Thus (\ref{equadg}) with $\delta P_i(0)$ replacing $e_i$ is just
(\ref{p0}). Finally, the first delooping relation follows
immediately from (\ref{cubic}) and (\ref{equadg}), while the other
two follow from (\ref{braidpr}).
\end{proof}

We wish to illuminate which relations are necessary in
addition to those which define $BMW_r(\CK)$, to obtain $E_q(2,r)$,
i.e. we wish to study $\Ker(\eta_q)$.
Note that the specialisation of Lemma \ref{bmwspec} (2) is the
classical limit as $q\to 1$ of $BMW_r(q)$, and that this is just the
Brauer algebra with parameter $\delta_{q\to 1}=3$ in accord with
\cite[\S 3]{LZ}. We shall study $\Ker(\eta_q)$ by first examining
the classical case, and then use specialisation arguments. 
The cellular structure of the algebras involved will play an important 
role in what follows. 

\subsection{Tensor notation and quantum action}\label{qnotation}
In this subsection we establish notation for basis elements of
tensor powers, which is convenient for computation of the actions we
consider. Since the $\fsl_2$-module $V(2)$ is the classical limit at
$q\to 1$ of $V_q(2)$, we do this for the quantum case, and later
obtain the classical one by putting $q=1$.

Maintaining the notation of section \ref{basic}, and proceeding
as in the proof of Lemma \ref{tl-alternator}, let $v_{-1}\in V_q(2)$
be a basis element of the $-2$ weight space, and let $e,f,k$ be the
generators of $\cU_q$ referred to in \S\ref{basic}. Then $v_0:=ev_{-1}$
and $v_{1} :=ev_0$ have weights $0,2$ respectively, and
$\{v_0,v_{\pm 1}\}$ is a basis of $V_q(2)$. Moreover it is easily
verified that $fv_1=(q+q\inv)v_0$ and $fv_0=(q+q\inv)v_{-1}$. The
tensor power $V_q(2)^{\otimes r}$ has a basis consisting of elements
$v_{i_1,i_2,\dots,i_r}:=v_{i_1}\otimes v_{i_2} \otimes \dots\otimes
v_{i_r}$. Note that $v_{i_1,i_2,\dots,i_r}$ is a weight element of
weight $2(i_1+\dots +i_r)$ for the action of $\cU_q(\fsl_2)$.

Now $V_q(2)^{\otimes 2}$ has a canonical decomposition
$V_q(2)^{\otimes 2}=L(0)_q\oplus L(2)_q\oplus L(4)_q$, where
$L(i)_q$ is isomorphic to $V(i)_q$ for all $i$. We shall give bases
of the three components, which consist of weight vectors.

\begin{lem}\label{bases}
The three components of $V_q(2)^{\otimes 2}$ have bases as follows.
\begin{enumerate}
\item $L(0)_q$: $v_{-1,1}-q^2v_{0,0}+q^2v_{1,-1}$.
\item $L(2)_q$: $v_{0,1}-q^2v_{1,0};\;v_{-1,1}-v_{1,-1}+(1-q^2)v_{0,0};\;
v_{-1,0}-q^2v_{0,-1}$.
\item $L(4)_q$: $v_{1,1};\;v_{0,1}+q^{-2}v_{1,0};\; v_{-1,1}+(1+q^{-2})v_{0,0}
+q^{-4}v_{1,-1};\;v_{-1,0}+q^{-2}v_{0,-1};$ $\\ \;v_{-1,-1}$.
\end{enumerate}

The corresponding statement for the classical case is obtained by
putting $q=1$ above.

The $R$-matrix $\check R=R_1$ acts on the three components above via
the scalars $q^{-4}, -q^{-2}$ and $q^2$ respectively.
\end{lem}

The proof is a routine calculation, which makes use of the fact that
several of the basis elements above are characterised by the fact
that they are annihilated by $e$ and/or $f$.

It is useful to record the action of the endomorphism $e_1$ of
$V_q(2)^{\otimes 2}$ (see Theorem \ref{surjbmwend}) on the basis
elements $v_{i,j}$.

\begin{lem}\label{e-action}
The endomorphism $e_1$ of $V_q(2)^{\otimes 2}$ acts as follows. Let
$w_0=-q^2v_{0,0}+q^2v_{1,-1}+v_{-1,1}\in L(0)_q$. Then
$$
e_1v_{i,j}=\left\{
\begin{array}{l l}
w_0, &\text{if }(i,j)=(1,-1),\\
q^{-2}w_0, &\text{if }(i,j)=(-1,1),\\
-q^{-2}w_0, &\text{if }(i,j)=(0,0),\\
0, &\text{if }i+j\neq 0.
\end{array}\right.
$$
\end{lem}
\begin{proof}
Since $e_1$ acts as $0$ on $L(2)_q$ and $L(4)_q$, and as $[3]_q$ on
$L(0)_q$, this follows from an easy computation with the bases in
Lemma \ref{bases}.
\end{proof}

The next result will be used in the next section.

\begin{lem}\label{e2iso}
The $\cU_q(\fsl_2)$-homomorphism $e_2:L(2)_q\otimes L(2)_q \lr
V_q(2)\otimes L(0)_q\otimes V_q(2)$ which is obtained by restricting
$e_2$ to $L(2)_q\otimes L(2)_q\subset V_q(2)^{\otimes 4}$, is an
isomorphism.
\end{lem}
\begin{proof}
Let $u_i$ be the weight vector of weight $2i$ of $L(2)_q$ which is
given in Lemma \ref{bases} ($i=0,\pm 1$). Then $L(2)_q\otimes
L(2)_q$ has basis $\{u_{i,j}:=u_i\otimes u_j\mid i,j=0,\pm 1\}$.
Similarly, $V_q(2)\otimes L(0)_q\otimes V_q(2)$ has basis
$\{x_{i,j}:=v_i\otimes w_0\otimes v_j\mid i,j=0,\pm 1\}$.

Now from Lemma \ref{e-action}, $e_2v_{i,j,a,b}=0$ unless $j+a=0$.
This fact may be used to easily compute $e_2u_{i,j}$ in terms of the
$x_{a,b}$. The resulting $9\times 9$ matrix of the linear map $e_2$
is then readily seen to have determinant $\pm
q^{-4}(q^2+q^{-2}-1)(q^4+1-q^{-2}+q^{-4})$, which is non-zero,
whence the result.
\end{proof}

\begin{cor}\label{e2classical}
The statement in Lemma \ref{e2iso} remains true in the classical
case ($q=1$).
\end{cor}

This is clear since the determinant in the proof of Lemma
\ref{e2iso} does not vanish at $q=1$.


\section{The radical of a cellular algebra}\label{radical}


In the next section we shall meet some cellular algebras which
are not semisimple. This section is devoted to proving some general
results about such algebras,  which we use below.
In this section only, we take $B=B(\Lambda,M,C,^*)$
to be any cellular algebra over a field $\bF$,
and prove some general results concerning its radical $\CR$.
These results may be of some interest independently of the
rest of this work. 
We assume that the reader has some acquaintance with the general
theory of cellular algebras (see \cite[\S\S 1-3]{GL96});
notation will be as is standard in cellular theory.
In particular, for any element $\lambda\in\Lambda$, the corresponding
cell module will be denoted by $W(\lambda)$ and its radical with respect
to the canonical invariant bilinear form $\phi_\lambda$ by $R(\lambda)$.
Since there is no essential loss of generality,
we shall assume for ease of exposition, that $B$ is quasi-hereditary.

Let $\lambda\in\Lambda$, and write $W(\lambda)^*$ for the
dual of $W(\lambda)$; this is naturally a right $B$-module, and
we have a vector space monomorphism (cf. \cite[(2.2)(i)]{GL96})
$C^\lambda:W(\lambda)\otimes_\bF W(\lambda)^*\lr B$ defined by
\be\label{clambda}
C^\lambda(C_S\otimes C_T)=C^\lambda_{S,T}\text{ for }S,T\in M(\lambda).
\ee

Denote the image of $C^\lambda$ by $B(\{\lambda\})$. This is a subspace
of $B$, isomorphic as $(B,B)$-bimodule to $B(\leq\lambda)/B(<\lambda)$
(see [{\it loc. cit.}]{GL96}), and we have a vector space isomorphism
\begin{equation}\label{add}
B\overset{\sim}{\lr}\oplus_{\lambda\in\Lambda}B(\{\lambda\}).
\end{equation}

Note that $W(\lambda)$ and $W(\lambda)^*$ are equal as sets. We shall
therefore differentiate between them only when actions are relevant.

Now let $\pi: B\lr \overline B:=B/\CR$ be the natural map
from $B$ to its largest semisimple quotient. Then
$\overline B\cong\oplus_{\lambda\in\Lambda}\overline B(\lambda)$,
where $\overline B(\lambda)\cong M_{l_\lambda}(\bF)\cong \End_\bF(L(\lambda))$.
Thus $\pi$ may be written $\pi=\oplus\pi_\lambda$, where
$\pi_\lambda:B\lr\End_\bF(L(\lambda))$ is the representation of $B$
on $L(\lambda)$. We collect some elementary observations in the next Lemma.

\begin{lem}\label{annihilators}
\begin{enumerate}
\item The restriction $\pi_\lambda:B(\{\lambda\})\lr \End_\bF(L(\lambda))$
is a surjective linear map for each $\lambda\in\Lambda$.
\item Let $A$ be any semisimple $\bF$-algebra and let $\sigma:B\lr A$ be a surjective
homomorphism. Then $\sigma$ factors through $\pi$ as shown.

\be\label{annihil}
\psmatrix[colsep=18mm,rowsep=15mm]
  B & &\overline B\cong\oplus_{\lambda}\overline B(\lambda)\\ & A
  \psset{nodesep=2pt}\everypsbox{\scriptstyle}
  \ncline{->}{1,1}{2,2}\tlput{\sigma}
  \ncline
  {->}{1,3}{2,2}\trput{\overline\sigma}
  \ncline{->}{1,1}{1,3}^{\pi}
\endpsmatrix
\ee
\item The restriction $\overline\sigma_\lambda$ of $\overline \sigma$
to $\overline B(\lambda)$ is either zero or an isomorphism.
\item Denote
by $\Lambda^0$ the set $\{\lambda\in\Lambda\mid \overline\sigma_\lambda
\text{ is non-zero}\}$. Let $J=\Ker(\sigma)$. Then $\pi_\lambda(J)=0$
if and only if $\lambda\in\Lambda^0$.
\item The radical $\CR$ is the set of elements of $B$ which act 
as zero on each irreducible module $L(\lambda)$.

\end{enumerate}
\end{lem}
\begin{proof}
The statement (1) follows from the cyclic nature of the cell
modules (\cite[(2.6)(i)]{GL96}). The only other statement
deserving of comment is (4), which follows immediately from
the observation that $A\cong B/J\cong\overline B/\overline J$,
where $\overline J=J/\CR$ is a two-sided ideal of the semisimple
algebra $\overline B$. The ideal $\overline J$ therefore acts
trivially in precisely those irreducible representations of
$B$ which ``survive'' in the quotient, and non-trivially in the others.
Note that (5) follows immediately from (4). 
\end{proof}

\begin{cor}\label{kereta} Let $\eta:B\lr \End_\bF(W)$ be a 
representation of $B$ in the semisimple $B$-module $W$, 
and write $E=\im(\eta)$.
Let $N=\Ker(\eta:B\lr E)$. Then $E\cong B/N
\cong\oplus_{\lambda\in\Lambda^0}\ov B(\lambda)$,
where $\Lambda^0$ is the set of $\lambda\in\Lambda$ such that $L(\lambda)$
is a direct summand of $W$, regarded as an $E$-module.
Moreover $\Lambda^0$
is characterised as the set of $\lambda\in \Lambda$ such that $N$
acts trivially on $L(\lambda)$.
\end{cor}

To study the action of ideals on the $L(\lambda)$ we shall require
the following results.

\begin{lem}\label{rad1} Assume that $B$ is quasi-hereditary; i.e. that
for all $\lambda\in\Lambda$, $\phi_\lambda\neq 0$.
\begin{enumerate}
\item In the notation of (\ref{clambda}), if $x,y$ and $z$ are
elements of $W(\lambda)$, then
$C^\lambda(x\otimes y)z=\phi_\lambda(y,z)x$.
\item If $x$ or $y$ is in $R(\lambda)$, then
$\pi_\lambda(C^\lambda(x\otimes y))=0$.
\item The radical $\CR$ of $B$ has a filtration
$(\CR(\lambda)=\CR\cap B(\leq\lambda))$ by two sided
ideals such that there is an isomorphism of $(B,B)$-bimodules
$$\CR(\lambda)/\CR(<\lambda)\overset{\sim}{\lr}
W(\lambda)\otimes R(\lambda)^*+R(\lambda)\otimes W(\lambda)^*
\subset W(\lambda)\otimes W(\lambda)^*.$$ 

\end{enumerate}
\end{lem}
\begin{proof}
(1) is just \cite[(2.4)(iii)]{GL96}. To see (2),
note that if $y\in R(\lambda)$, then from (1),
$C^\lambda(x\otimes y)z=\phi_\lambda(y,z)x=0$ for all
$z\in W(\lambda)$. If $x\in R(\lambda)$, then
again by (1), $C^\lambda(x\otimes y)z\in R(\lambda)$,
whence $C^\lambda(x\otimes y)$ acts as $0$ on
$L(\lambda)=W(\lambda)/R(\lambda)$.

Now suppose
$b\in C^\lambda(W(\lambda)\otimes R(\lambda)+R(\lambda)
\otimes W(\lambda))$. Then by (2), $\pi_\lambda(b)=0$.
We shall prove

\be\label{recursive}
 \exists\;\text{ elements } b_{\lambda'}\in B(\{\lambda'\})
\;(\lambda'<\lambda)\text{ such that }\\
b+\sum_{\lambda'<\lambda}b_{\lambda'}\in\CR.
\ee

We do this recursively as follows. Suppose we have a subset
$\Gamma\subseteq\Lambda$ and an element $\sum_{\gamma\in\Gamma}b_\gamma
\in\sum_{\gamma\in\Gamma}B(\{\gamma\})$ such that
for any $\beta\in\Lambda$ which satisfies $\beta\geq\gamma$ for some
$\gamma\in\Gamma$, we have $\pi_\beta(\sum_{\gamma\in\Gamma}b_\gamma)=0$.
We show that if there is $\mu\in\Lambda$ such that
$\pi_\mu(\sum_{\gamma\in\Gamma}b_\gamma)\neq 0$, then we may increase
$\Gamma$ to obtain another set with the same properties.
For this, take $\mu\in\Lambda$ such that
$\pi_\mu(\sum_{\gamma\in\Gamma}b_\gamma)\neq 0$, and maximal with respect to this
property. Note that since $\pi_\beta(b_\gamma)\neq 0$ implies that
$\gamma\geq\beta$, we have $\mu\leq\gamma$ for some element
$\gamma\in\Gamma$. By Lemma \ref{annihilators}(1), there is an element
$b_\mu\in B(\{\mu\})$ such that
$\pi_\mu(\sum_{\gamma\in\Gamma}b_\gamma)=\pi_\mu(-b_\mu)$.
Let $\Gamma'=\Gamma\cup\{\mu\}$. If $\beta\geq\gamma'$ for some
$\gamma'\in\Gamma$, we show that
$\pi_\beta(\sum_{\gamma'\in\Gamma'}b_{\gamma'})= 0$.

There are two cases. If $\gamma'\in\Gamma$, then
$\pi_\beta(\sum_{\gamma\in\Gamma}b_\gamma)= 0$. If
$\pi_\beta(b_\mu)\neq 0$, then $\beta\leq \mu$ and so
$\gamma'\leq \beta\leq\mu$, whence by hypothesis
$\pi_\mu(\sum_{\gamma\in\Gamma}b_\gamma)= 0$, a contradiction.
Hence $\pi_\beta(b_\mu)=0$, which proves the assertion in this case.

The remaining possibility is that $\gamma'=\mu$. In this case,
since $\pi_\mu(\sum_{\gamma'\in\Gamma'}b_{\gamma'})= 0$ by
construction, we may suppose $\beta>\mu$. But then by the maximal
nature of $\mu$,
$\pi_\beta(\sum_{\gamma\in\Gamma}b_\gamma)= 0$. Moreover
since $\beta>\mu$, $\pi_\beta(b_\mu)=0$. Hence
$\Gamma'$ and $\sum_{\gamma'\in\Gamma'}b_{\gamma'}$
have the same property as $\Gamma$ and $\sum_{\gamma\in\Gamma}b_\gamma$.
Note that $\Gamma'$ is obtained from $\Gamma$ by adding an element $\mu$
such that $\mu\leq\gamma$ for some $\gamma\in\Gamma$.

Now to prove the assertion (\ref{recursive}), start with 
$\Gamma=\{\lambda\}$ and $b_\lambda=b$. The argument above 
shows that we may repeatedly add elements $\mu<\lambda$ to $\Gamma$,
with corresponding $b_\mu\in B(\{\mu\})$, eventually coming to
a set $\Gamma_{max}$ such that $\sum_{\mu\in\Gamma_{max}}b_\mu$
acts trivially on each $L(\beta)$ ($\beta\in\Lambda$).

This completes the proof of (\ref{recursive}), and hence of
(3).
\end{proof}

The arguments used in the proof of the above Lemma may be applied
to yield the
following result, in which we use the standard notation of \cite{GL96}
for cellular theory.

\begin{thm}\label{rad}
Let $B=(\Lambda,M,C,*)$ be a cellular algebra over a field $\bF$,
and assume that $B$ is quasi-hereditary, i.e. that the invariant form
$\phi_\lambda$ on each cell module is non-zero.
For $\lambda\in\Lambda$, denote by $W(\lambda)$ and $R(\lambda)$
respectively the corresponding cell module and its radical.
\begin{enumerate}
\item Let $\lambda\in\Lambda$ and take any elements
$x\in W(\lambda)$, $z\in R(\lambda)$. Then there exist elements
$r(x,z)\in C^\lambda(x\otimes z)+B(< \lambda)$ and
$r(z,x)\in C^\lambda(z\otimes x)+B(< \lambda)$, both in $\CR$,
the radical of $B$.
\item Let $X$ be a subset of $B$ such that for
all $\lambda\in\Lambda$, $x\in W(\lambda)$ and $z\in R(\lambda)$,
$X$ contains elements $r(x,z)$ and $r(z,x)$ as in (1).
Then the linear subspace of $B$ spanned by
$X$ contains $\CR$.
\item Suppose $J$ is a two-sided ideal of $B$ such that
$J^*=J$. Let $\Lambda^0:=\{\lambda\in\Lambda\mid JL(\lambda)=0\}$.
Then $J\supseteq \CR$ if and only if,
for all $\lambda\in\Lambda^0$,
$R(\lambda)\subseteq JW(\lambda)$.
\end{enumerate}
\end{thm}
\begin{proof} The argument given in the proof of Lemma \ref{rad1}(3)
proves the statement (1).

For each $\lambda\in\Lambda$, let $w_\lambda=\dim\;W(\lambda)$,
$r_\lambda=\dim R(\lambda)$, and
$l_\lambda=\dim L(\lambda)=\dim\left(W(\lambda)/R(\lambda)\right)
=w_\lambda-r_\lambda$.
Then
$$
\begin{aligned}
\dim_\bF(\CR)&=\dim_\bF(B)-\sum_{\lambda\in\Lambda}l_\lambda^2\\
&=\sum_{\lambda\in\Lambda}w_\lambda^2-\sum_{\lambda\in\Lambda}l_\lambda^2\\
&=\sum_{\lambda\in\Lambda}r_\lambda(w_\lambda+l_\lambda).\\
\end{aligned}
$$
But it is evident from an easy induction 
 in the poset $\Lambda$ that
the dimension of the space spanned by the elements $r(x,z)$
and $r(z,x)$ is at least equal to
$$
\begin{aligned}
&\sum_{\lambda\in\Lambda}(2w_\lambda r_\lambda-r_\lambda^2)\\
=&\sum_{\lambda\in\Lambda}r_\lambda(2w_\lambda -r_\lambda)\\
=&\sum_{\lambda\in\Lambda}r_\lambda(w_\lambda +l_\lambda).\\
\end{aligned}
$$
Comparing with $\dim(\CR)$, we obtain the statement (2).

We now turn to (3). We begin by showing
\be\label{whole}
J\supseteq\CR\iff R(\lambda)\subseteq JW(\lambda) 
\text{ for all }\lambda\in\Lambda.
\ee

First assume $J\supseteq\CR$ and suppose $z\in R(\lambda)$;
then take $x,y\in W(\lambda)$, such that $\phi_\lambda(x,y)\neq 0$.
Since $\CR$, and therefore $J$, contains an element $r(z,x)$ of
the form above, we have $JW(\lambda)\ni r(z,x)y=\phi_\lambda(x,y)z$.
Hence $R(\lambda)\subseteq JW(\lambda)$ for each $\lambda\in\Lambda$.

Conversely, suppose $JW(\lambda)\supseteq R(\lambda)$
for each $\lambda\in\Lambda$. Since $J^*=J$, to show that
$J\supseteq \CR$, it will suffice to show that for any $\lambda\in\Lambda$,
and $x\in W(\lambda), z\in R(\lambda)$, there is an element
$r(z,x)\in J$, of the form above. Now by hypothesis, $z\in JW(\lambda)$;
hence $C^\lambda(z\otimes x)=C^\lambda(ay\otimes x)\in
aC^\lambda(y\otimes x)+B(<\lambda)$, for some $a\in J$ and
$y\in W(\lambda)$. Hence there is an element
$a_1=C^\lambda(z\otimes x)+b\in J$ where
$b\in B(<\lambda)$. If $a_1\not\in\CR$,
then there is an element $\lambda'<\lambda$
such that $a_1L(\lambda')\neq 0$, since
$a_1L(\mu)=0$ for all $\mu\not<\lambda$.

By the cyclic nature of
cell modules, if $JL(\mu)\neq 0$, then $JW(\mu)=W(\mu)$.
Thus for any two elements $p,q\in W(\mu)$, since $p\in JW(\mu)$, 
the argument above shows that
$C^\mu(p\otimes q)+b'\in J$ for some $b'\in B(<\mu)$.
It follows that
the argument in the proof of Lemma \ref{rad1} may be applied to show that
$a_1$ may be recursively modified by elements of $J$, to yield
an element $a_0=r(z,x)\in J\cap \CR$ as required.

The statement (\ref{whole}) now follows from (2). To deduce (3),
observe that if $\lambda\in\Lambda\setminus \Lambda^0$, then since
$JL(\lambda)\neq 0$, there are elements $a\in J$ and $x\in W(\lambda)$
such that $ax\not\in R(\lambda)$. But then by Lemma \ref{rad1}(1),
$W(\lambda)=B\cdot ax\subseteq J\cdot x\subseteq JW(\lambda)$, whence
$JW(\lambda)\supseteq R(\lambda)$ always holds {\it a fortiori}
for $\lambda\in\Lambda\setminus\Lambda^0$. In view of (\ref{whole})
this completes the proof of (3).
\end{proof}

\begin{cor}\label{corirred}
Let notation be as in Theorem \ref{rad}. Assume that for all $\lambda\in\Lambda$
such that $JL(\lambda)=0$, $R(\lambda)$ is either zero or irreducible. Assume further
that for any $\lambda\in\Lambda$ such that $JL(\lambda)=0$ and $R(\lambda)\neq 0$,
$JW(\lambda)\neq 0$. Then $J$  contains the radical $\CR$ of $B$.
\end{cor}

\begin{proof}
It follows from Theorem \ref{rad} that it suffices to show that for any
$\lambda$ such that $JL(\lambda)=0$, $JW(\lambda)=R(\lambda)$. But by
hypothesis, if $R(\lambda)\neq 0$ for some such $\lambda$, $JW(\lambda)$
is a non-zero submodule of $R(\lambda)$. By irreducibility, it follows that
$JW(\lambda)=R(\lambda)$, whence the result.
\end{proof}

\section{The classical 3-dimensional case}\label{classical3}

\subsection{The setup}\label{setup-class}
Let $V=V(2)$, the classical three-dimensional irreducible
representation of $\fsl_2(\C)$. In this section we shall construct a
quotient of the Brauer algebra $B_r(3)$ which is defined by adding a
single relation to the defining relations of $B_r(3)$, and which
maps surjectively onto $\End_{\fsl_2}(V^{\otimes r})$. For small $r$
we are able to show that our quotient is isomorphic to the
endomorphism algebra. We shall make extensive use of the cellular
structure of $B_r(3)$, as outlined in \cite[\S 4]{GL96}. In analogy
with the Temperley-Lieb case above, where the case $r=3$ (i.e.
$V(1)^{\otimes 3}$) was critical, we start with the case $r=4$, i.e.
$V(2)^{\otimes 4}$.

Recall that given a commutative ring $A$, the Brauer algebra
$B_r(\delta)$ over $A$ may be defined as follows. It has generators
$\{s_1,\dots,s_{r-1};e_1,\dots,e_{r-1}\}$, with relations
$s_i^2=1,\; e_i^2=\delta e_i,\; s_ie_i=e_is_i=e_i$ for all $i$,
$s_is_j=s_js_i,\;s_ie_j=e_js_i,\;e_ie_j=e_je_i$ if $|i-j|\geq 2$,
and $s_is_{i+1}s_i=s_{i+1}s_is_{i+1},\; e_ie_{i\pm 1}e_i=e_i$ and
$s_ie_{i+1}e_i=s_{i+1}e_i,\; e_{i+1}e_is_{i+1}=e_{i+1}s_{i}$ for all
applicable $i$. We shall assume the reader is familiar with the
diagrammatic representation of a basis of $B_r(\delta)$, and how
basis elements are multiplied by concatenation of diagrams. In
particular, the group ring $A\Sym_r$ is the subalgebra of
$B_r(\delta)$ spanned by the diagrams with $r$ ``through strings'',
and the algebra contains elements $w\in\Sym_r$ which are appropriate
products of the $s_i$.

In this section we take $A=\C$ and $\delta=3$. The algebra $B_r(3)$
acts on $V^{\otimes r}$ as follows. We take the same basis
$\{v_i\mid i=0,\pm 1\}$ as in \S \ref{qnotation}. Then $s_i$ acts by
interchanging the $i^{\text{th}}$ and $(i+1)^{\text{st}}$ factors in
the tensor $v_{j_1}\otimes\dots\otimes v_{j_r}$. We define $w_0\in
V\otimes V$ as the specialisation at $q=1$ of the element $w_0$ of
\S \ref{qnotation}, i.e. $w_0=v_{-1,1}+v_{1,-1}-v_{0,0}$. Then the
action of $e_1$ is obtained by putting $q=1$ in Lemma
\ref{e-action}, i.e.
$$
e_1v_{i,j}=\left\{
\begin{array}{l l}
w_0, &\text{if }(i,j)=(1,-1)\text{ or }(-1,1),\\
-w_0, &\text{if }(i,j)=(0,0),\\
0, &\text{if }i+j\neq 0.
\end{array}\right.
$$

The element $e_i$ acts on the $i,i+1$ components similarly. In
addition to the elements $s_i$ and $e_i$, it will be useful to
define the endomorphisms $e_{i,j}:=(1,i)(2,i+1)e_1(1,i)(2,i+1)$,
where we use the usual cycle notation for permutations in
$B_r(\delta)$. The endomorphism $e_{i,j}$ acts on the $i^{\text{th}}$
and $j^{\text{th}}$ components of $V^{\otimes r}$ as $e_1$, and
leaves the other components unchanged.

\subsection{Cellular structure}\label{cell-br} The Brauer algebra
$B_r(\delta)$ was proved in \cite[\S 4]{GL96} to have a cellular
structure. This facilitates discussion of its representation theory.
We begin by reviewing
briefly the cells and cell modules for $B_r(\delta)$. Our notation
here differs slightly from that in {\it loc. cit.}.

Given an integer $r\in\Z_{\ge 0}$, define $\CT(r): =\{t\in \Z\mid
0\leq t\leq r,\;\text{ and } r-t\in 2\Z\}$. For $t\in\Z_{\geq 0}$,
let $\CP(t)$ denote the set of partitions of $t$. Define
$\Lambda(r):=\amalg_{t\in\CT(r)}\CP(t)$. This set is partially
ordered by stipulating that $\lambda> \lambda'$ if
$|\lambda|>|\lambda'|$ or $|\lambda|=|\lambda'|$, and
$\lambda>\lambda'$ in the dominance order on partitions of
$|\lambda|$. For any partition
$\lambda=(\lambda_1\geq\lambda_2\geq\dots\geq \lambda_p)$, denote by
$|\lambda|$ the sum $\sum_i\lambda_i$ of its parts. Given
$\lambda\in\Lambda(r)$, the corresponding set $M(\lambda)$ (cf.
\cite[(1.1) and \S 4]{GL96}) is the set of pairs $(S,\tau)$, where
$S$ is an involution with $|\lambda|$ fixed points in $\Sym_{r}$
and $\tau$ is a standard tableau of shape $\lambda$.

If $(S,\tau)$ and $(S',\tau')$ are two elements of $M(\lambda)$,
the basis element of $B_r(\delta)$ is, in the
notation of \cite[(4.10)]{GL96},
$$
C^\lambda_{(S,\tau),(S',\tau')}:=
\sum_{w\in \Sym_{|\lambda|}}p_{w(\tau,\tau')}(w)[S,S',w],
$$
where $w(\tau,\tau')$ is the element of $\Sym_{|\lambda|}$ corresponding
to $\tau,\tau'$ under the Robinson-Schensted correspondence, and for
$u\in \Sym_{|\lambda|}$,
$c_u=\sum_{w\in\Sym_{|\lambda|}}p_u(w)w$ is the corresponding
Kazhdan-Lusztig basis element of $\Z\Sym_{|\lambda|}$.
The cardinality $|M(\lambda)|$ is easily computed. Let
$k(\lambda)=\frac{r-|\lambda|}{2}$, and let $d_\lambda$ be the
dimension of the representation (Specht module) of the symmetric
group $\Sym_{|\lambda|}$ corresponding to $\lambda$. For any integer
$t\geq 0$ denote by $t!!$ the product of the odd positive integers
$2i+1\leq t$. Then we have, for any $\lambda\in\Lambda(r)$,
\begin{equation}\label{cardmlambda}
|M(\lambda)|=\binom{r}{|\lambda|}(2k(\lambda))!!d_\lambda:=w_\lambda.
\end{equation}

Now assume that the ground ring $A$ is a field. We recall some facts
from cellular theory.
\begin{assertion}\label{repsbr} Maintain the notation above.
\begin{enumerate}
\item For each $\lambda\in\Lambda(r)$, there is a
left $B_r(\delta)$-module $W(\lambda)$, of dimension $w_\lambda$
over $A$.
\item The module $W(\lambda)$ has a bilinear form
$\phi_\lambda:W(\lambda)\times W(\lambda)\to A$, which is invariant
under the $B_r(\delta)$-action.
\item Let $R(\lambda)$ be the radical of the form $\phi_\lambda$.
Then $L(\lambda):=W(\lambda)/R(\lambda)$ is either an irreducible
$B_r(\delta)$-module, or is zero. The non-zero $L(\lambda)$ are
pairwise non-isomorphic, and all irreducible $B_r(\delta)$-modules
arise in this way.
\item All composition factors $L(\mu)$ of $W(\lambda)$
satisfy $\mu\geq \lambda$.
\item $B_r(\delta)$ is semisimple if and only if each
form $\phi_\lambda$ is non-degenerate. Equivalently, the $W_\lambda$
form a complete set of representatives of the isomorphism classes of
simple $B_r(\delta)$-modules.
\end{enumerate}
\end{assertion}

We shall make use of the following facts.

\begin{prop}\label{brss}Take $A=\C$ and assume that $\delta\neq 0$.
\begin{enumerate}
\item The algebra $B_r(\delta)$ is quasi-hereditary; that is,
each form $\phi_\lambda$ in the assertion \ref{repsbr}(2) is non-zero.
\item The algebra $B_r(3)$ is semisimple if and only if
$r\leq 4$.
\end{enumerate}
\end{prop}
\begin{proof}
The statement (1) is immediate from \cite[Corollary (4.14)]{GL96},
while (2) follows from \cite[Theorem 2.3]{RS}.
\end{proof}

\subsection{The case $r=4$}\label{4} It is clear from dimension
considerations that when $r\leq 3$, the surjection $\eta: B_r(3)\lr
\End_{\fsl_2}V(2)^{\otimes r}$ (and its quantum analogue) is an
isomorphism. The case $r=4$ is therefore critical. In this
subsection we shall treat the classical case when $r=4$.

In terms of \S \ref{setup-class}, we now take $r=4$ and $\delta=3$.
Our purpose is to identify the kernel of the surjection
$\eta:B_4(3)\lr \End_{\fsl_2(\C)}V^{\otimes 4}$. Define the element
$\Phi\in B_4(3)$ by
\begin{equation}\label{defPhi}
\begin{aligned}
\Phi&=Fe_2F-F-\frac{1}{4}Fe_2e_{1,4}F, \text { where }\\
F&=(1-s_1)(1-s_3),\\
\end{aligned}
\end{equation}
where notation is as in \S \ref{setup-class}.

The next statement summarises some of the properties of $\Phi$.

\begin{prop}\label{propsPhi}
Let $F,\Phi\in B_4(3)$ be the elements defined in (\ref{defPhi}).
Then
\begin{enumerate}
\item $e_i\Phi=0$ for $i=1,2,3$.
\item $\Phi^2=-4\Phi$.
\item $\Phi$ acts as $0$ on $V^{\otimes 4}$. That is, $\Phi\in\Ker(\eta)$.
\end{enumerate}
\end{prop}

\begin{proof}
First note that
\begin{equation}\label{eFe}
e_2Fe_2=e_2+e_2e_{1,4}.
\end{equation}
To see this, observe that $e_2Fe_2=e_2(1-s_1-s_3+s_1s_3)e_2
=e_2^2-e_2s_1e_2-e_2s_3e_2+e_2s_1s_3e_2=3e_2-2e_2+e_2e_{1,4}$.

To prove (1), note that it is trivial that $e_iF=0$ for $i=1,3$, and
hence that $e_i\Phi=0$ for $i=1,3$. But the relation $e_2\Phi=0$ now
follows easily from (\ref{eFe}) and the fact that
$e_{1,4}^2=3e_{1,4}$.

It follows from (1) that $e_2F\Phi=0$, since $F\Phi=4\Phi$ (recall
$F^2=4F$). Hence $\Phi^2=-F\Phi=-4\Phi$ which proves (2).

For (3), note that $F$ maps $V^{\otimes 4}$ onto $L(2)\otimes L(2)$,
and hence that $\Phi(V^{\otimes 4}) \subseteq L(2)\otimes L(2)$. But
by Corollary \ref{e2classical}, $e_2$ acts injectively on
$L(2)\otimes L(2)$, whence it follows from the fact that
$e_2\Phi=0$, just proved, that $\Phi(V^{\otimes 4})=0$.
\end{proof}

\begin{thm}\label{case4}
The kernel of $\eta: B_4(3)\lr \End_{\fsl_2}V(2)^{\otimes 4}$ is
generated by the element $\Phi$ above.
\end{thm}

\begin{proof}
The set $\Lambda(4)$ has $8$ elements, ordered as follows:
$$
(4)>(3,1)>(2^2)>(2,1^2)>(1^4)>(2)>(1^2)>(0).
$$

The dimensions of the corresponding cell modules, which by
the assertion \ref{repsbr}(5) and Proposition \ref{brss}(2) 
are simple in this case, are
given respectively by
$$
1,3,2,3,1,6,6,3.
$$

Now since $B_4(3)$ is semisimple, it is isomorphic to a sum
$\oplus_{j=1}^8M(j)$ of 2-sided ideals, which are isomorphic to
matrix algebras of size given in the list above (thus, e.g., $\dim
M(1)=1$, while $\dim M(7)=36$). Moreover the 2-sided ideal $\CI$ of
$B_4(3)$ which is generated by the $e_i$ is cellular, and is the sum
of the matrix algebras $M(j)$ for $j\geq 6$.

Note that $B_4(3)/\CI\cong \C\Sym_4$. Let $\CP$ be the 2-sided ideal
of $B_4(3)$ generated by $\Phi$. Then $\CP+\CI\ni F$, and
$\frac{1}{4}F$ is an idempotent in $\C\Sym_4$ which generates a left
ideal on which $\Sym_4$ acts as $\Ind_{K}^{\Sym_4}(\varepsilon)$,
where $K$ is the subgroup of $\Sym_4$ generated by $s_1,s_3$ and
$\varepsilon$ denotes the alternating representation. But it is
easily verified that $\Ind_{K}^{\Sym_4}(\varepsilon)$ is isomorphic
to the sum of the irreducible representations of $\Sym_4$ which correspond to
the partitions $(2^2), (2,1^2)$ and $(1^4)$, each one occurring with
multiplicity one. Here we use the standard parametrisation in which
the irreducible complex representations of $\Sym_n$ correspond to 
partitions of $n$, the trivial representation corresponding to the 
partition $(n)$.

It follows that the 2-sided ideal of $\C\Sym_4$
generated by $F$ is the image of $\oplus_{j=3}^5M(j)$ under the
surjection $B_4(3)/\CI\lr \C\Sym_4$. It follows that
$\CI+\CP=\oplus_{j\geq 3}M_j$, whence $\dim(\CI+\CP)=\dim\CI+14$.

But using the dimension formula for $\dim\End_{\fsl_2}V^{\otimes 4}$
in (\ref{d2r}), the kernel $N$ of $\eta$ has dimension $14$ in this
case. Since $\CP\subseteq N$, it follows that
$$
\dim (\CI+\CP)\leq\dim(\CI+N)\leq\dim\CI+\dim\CP\leq \dim\CI+\dim N
=\dim\CI+14,
$$
with equality if and only if $\CI\cap N=0$ and $\CP=N$.

Since we have proved equality, the theorem follows.
\end{proof}

\subsection{The case $r=5$}\label{5} This is the first case where
$B:=B_r(3)$ is not semisimple. We shall analyse this case to show how
our methods yield non-trivial information on the algebras, such as
the dimension of the radical. For this subsection only, we denote
$B_5(3)$ by $B$.

In this case the cells are again totally ordered; we write them as
follows.
\begin{eqnarray}\label{cells}
\begin{aligned}
(5)>(4,1)>(3,2)>(3,1^2)>(2^2,1)>(2,1^3)>(1^5)\\ >(3)>(2,1)
>(1^3)>(1).
\end{aligned}
\end{eqnarray}

If $W(\lambda)$ denotes the cell module corresponding to $\lambda$,
the dimensions of the $W(\lambda)$ above are respectively given by:
$$
1,4,5,6,5,4,1,10,20,10,15.
$$
Recall that  $L(\lambda)$ is the irreducible head of $W(\lambda)$
for $\lambda\in\Lambda(5)$;
write $l_\lambda:=\dim L(\lambda)$. These integers are the dimensions
of the simple $B$-modules.

We define the following 2-sided ideals of $B$. Let $\CR$ be the
radical of $B_5(3)$, $\CI=B(\leq(3))$ the ideal generated by
the $e_i$, $\CP$ the ideal generated by $\Phi$, and $N$ the kernel
of $\eta:B\lr E:=\End_{\fsl_2}(V^{\otimes 5})$.

We shall prove

\begin{thm}\label{r5}
Let $B=B_5(3)$ as above, let $\CR$ be its radical, and maintain the
above notation.
\begin{enumerate}
\item The cell modules of $B$ are all simple except for those
corresponding to the partitions
$(2, 1)$ and $(1^3)$, whose simple heads have dimension $15, 6$
respectively.
\item The composition factors of $W(1^3)$ are $L(1^3)$ and
$L(2,1^3)$.
\item The composition factors of $W(2,1)$ are $L(2,1)$ and
$L(2^2,1)$.
\item The radical of $B$ has dimension $239$.
\item The kernel $N$ of $\eta:B\lr E$ is generated by
$\Phi$ modulo the radical. That is, in the notation above,
$N=\CP+\CR$.
\end{enumerate}
\end{thm}
\begin{proof}
It is easily verified that $E\cong M_1(\C)\oplus M_{4}(\C)\oplus
M_{10}(\C) \oplus M_{15}(\C)\oplus M_{15}(\C)\oplus M_{6}(\C)$,
where $M_n(\C)$ denotes the algebra of matrices of size $n$ over
$\C$. Further, $\eta$ induces a surjection $\overline\eta:\overline
B:=B/\CR\cong\oplus_{\lambda\in\Lambda} M_{l_\lambda}(\C)\lr E$.
We shall determine which of the simple components $M_{l_\lambda}(\C)$
are in the kernel of $\overline\eta$.

Observe that since $B/\CI\cong\C\Sym_5$, the cell modules
$W(\lambda)$ ($|\lambda|=5$) are all irreducible, and clearly
$\CR\subseteq \CI\cap N$.

Now the element $F=(1-s_1)(1-s_3)$ generates the 2-sided ideal of
$\C\Sym_5$ which corresponds to the irreducible represesentations of
$\Sym_5$ which are constituents of $\Ind_K^{\Sym_5}(\varepsilon)$,
where $K$ is the subgroup generated by $s_1$ and $s_3$. An easy
computation shows that these representations are precisely those
which correspond to the partitions $\lambda$ with $|\lambda|=5$
and $\lambda\neq (5),(4,1)$. Let $\Lambda^1=
\{\lambda\in\Lambda\mid \;|\lambda|=5,\lambda\neq (5),(4,1)\}$,
and write $\Lambda^0:=\Lambda\setminus\Lambda^1$.
It follows from the above that $N$ acts non-trivially on the
simple modules $W(\lambda)$
for $\lambda\in\Lambda^1$ (since $\Phi\in N$ does), and hence that
$\Ker(\overline\eta)\supseteq \oplus_{\lambda\in\Lambda^1}
M_{l_\lambda}(\C)$. Using the number and
dimensions of the matrix components of $E$, it follows, by
comparing the
sizes of the matrix algebras on both sides, that $W(1)$
is simple, one of the $10$ dimensional cell modules is simple,
the other has head of dimension $6$, and
$W(2,1)$ has head of dimension $15$.

Now the Gram matrix associated with
the bilinear form $\phi_{(1^3)}$ on $W(1^3)$ is given by
\[
\begin{bmatrix}
 3& 1&-1& 1& 1&-1& 1& 0& 0& 0\\
 1& 3& 1&-1& 1& 0& 0&-1& 1& 0\\
-1& 1& 3& 1& 0& 1& 0&-1& 0& 1\\
 1&-1& 1& 3& 0& 0& 1& 0&-1& 1\\
 1& 1& 0& 0& 3& 1&-1& 1&-1& 0\\
-1& 0& 1& 0& 1& 3& 1& 1& 0&-1\\
 1& 0& 0& 1&-1& 1& 3& 0& 1&-1\\
 0&-1&-1& 0& 1& 1& 0& 3& 1& 1\\
 0& 1& 0&-1&-1& 0& 1& 1& 3& 1\\
 0& 0& 1& 1& 0&-1&-1& 1& 1& 3\\
\end{bmatrix}.
\]
Since this has rank $6$, $W(1^3)$ is reducible. To understand the composition
factors, note that the cell corresponding to the partition $1^3$
contains just the longest element $w_0$ of $\Sym_3$.
The Kazhdan-Lusztig basis element
$c_{w_0}=\sum_{w\in\Sym_3}\varepsilon(w)w$, and from this
one sees easily that the element $\sum_{w\in\langle s_1,s_2,s_3\rangle}w$
of $B$ acts trivially on $W(1^3)$, whence it follows that
$W(1^3)$ has no submodule isomorphic to $L(5)$ or $L(4,1)$.
Similarly, since $E(5)=\sum_{w\in\Sym_5}\varepsilon(w)w$ also
acts trivially on $W(1^3)$ (note that $E(5)e_i=0$ for all $i$),
$W(1^3)$ has no submodule isomorphic to $L(1^5)$. It follows that
$R(1^3)\cong L(4,1)$, proving (2)

Now consider $W(2,1)$. The corresponding cell of $\Sym_3$
contains the elements $r_1,r_2,r_1r_2$ and $r_2r_1$, where the simple
generators of $\Sym_3$ are denoted $r_1,r_2$ to avoid confusion with
$s_1,s_2\in B$. The corresponding Kazhdan-Lusztig basis elements
are $1-r_1$, $1-r_2$, $(1-r_1)(1-r_2)$ and $(1-r_2)(1-r_1)$.
In analogy with the previous case, one now verifies easily
that $\sum_{w\in\langle s_1,s_2,s_4\rangle}w$ and $E(5)$ act
trivially on $W(2,1)$ whence the latter cell module has no
simple submodule isomorphic to $L(5),L(4,1),L(3,2)$ or $L(1^5)$,
It follows by dimension that $R(3,2)\cong L(2^2,1)$,
completing the proof of (1), (2) and (3).

Clearly $\dim\CR=\dim B-\dim\overline B=20^2+10^2-(15^2+6^2)=239$,
which proves (4).

To prove (5) observe that since $F\in (\CP+\CI)$, the argument
concerning induced representations above shows that
$B/(\CP+\CI)\cong M_1(\C)\oplus M_4(\C)$. Hence
$$
\dim(\overline{\CP+\CI})\leq\dim(\overline N+\overline \CI)
\leq\dim\overline N+\dim\overline \CI,
$$
with equality if and only if $\overline N\cap\overline \CI=0$ and
$\overline{\CP+\CR}=\overline N$. But $\dim \overline{\CP+\CI} =\dim
\overline B-17$, and by the above argument, this is equal to
$\dim\overline N+\dim\overline \CI$, whence
$\overline{\CP+\CR}=\overline N$, i.e. $N=\Phi+\CR$.
\end{proof}

Note that it is possible that the methods of \cite{HW} could be applied
to give alternative proofs of some parts of Theorem \ref{r5}.

\subsection{The general classical case}\label{genclass}

Our objective in this section is to check some cases of
our main conjecture below, and reduce it to a specific question about
the action of $\Phi$ on certain cell modules of $B_r(3)$.
To do this we shall utilise the general results of
\S \ref{radical} above
about the radical of a cellular algebra.


In view of the results of the last subsection we make the
\begin{assertion}\label{conj}
\noindent{\bf Conjecture.} Let $B=B_r(3)$,
$E=\End_{\fsl_2(\C)}(V(2)^{\otimes r})$
and $\eta:B\lr E$ the natural surjection discussed above.
The kernel $N$ of $\eta$ is generated by the element
$\Phi=Fe_2F-F-\frac{1}{4}Fe_2e_{14}F\in B$.
\end{assertion}

To make use of the theory in the last section,
we shall develop more detail concerning the cellular structure of $B_r$.
We maintain the notation
above. In particular $\CR$ denotes the radical
of $B$, $\CI$ is the two-sided ideal generated by the $e_i$
and $\CP$ denotes the ideal generated by $\Phi$.

We start with the following elementary
observation.

\begin{lem}\label{sym}
Let $t\geq 4$ be an integer and consider the
symmetric group $\Sym_t$ generated by simple transpositions
$s_1,\dots,s_{t-1}$. The two-sided ideal of $\C\Sym_t$
generated by $F=(1-s_1)(1-s_3)$ is the sum of the minimal ideals
corresponding to all partitions with at least $4$ boxes in the first two
columns.
\end{lem}

This is an easy exercise, which may be proved by induction on $t$.

\begin{thm}\label{lambda0}
Let $\eta: B_r(3)\lr E:=\End_{\fsl_2(\C)}(V(2)^{\otimes r})$ be the surjection
discussed above, and let $N=\Ker(\eta)$. Define $\Lambda^0\subseteq\Lambda$ by
$\Lambda^0=\{(t), (t-1,1),1^3\mid 0\leq t\leq r;\;\;
t\equiv r(\text{mod $2$})\}$, and let $\Lambda^1:=\Lambda\setminus\Lambda^0$.
Let $\Phi$ be the element of $B=B_r$ defined above.
Then
\begin{enumerate}
\item For $\lambda\in\Lambda^1$, there is an element
$x_\lambda\in L(\lambda)$ such that $\Phi x_\lambda\neq 0$.
\item $N$ acts trivially on $L(\lambda)$
if and only if $\lambda\in\Lambda^0$
\item $E\cong\oplus_{\lambda\in\Lambda^0}\ov B(\lambda)$.
\item If $\CP$ denotes the ideal of $B$ generated by $\Phi$,
we have $\CP+\CR=N$, where $\CR$ is the radical of $B$.
\end{enumerate}
\end{thm}
\begin{proof}
Take $\lambda\in\Lambda^1$. If $t=|\lambda|\geq 4$,
consider the subalgebra of $B$ generated by the elements
$\{s_ie_{t+1}e_{t+3}\dots e_{t+2k-1}\mid 1\leq i\leq t-1\}$,
where $r=t+2k$. This is isomorphic to $\C\Sym_t$, and
$\Phi$ acts on the corresponding cell modules as $-F$.
The statement (1) is now clear for this case, given Lemma
\ref{sym}.

Now suppose $t=|\lambda|\leq 3$. Now in analogy to the
above argument, we consider the ``leftmost'' part of the
diagrams, completed with $e_{5}e_{7}\dots$
or $e_6e_8\dots$ on the right according as $t$ is odd or even.
The cases $r=4,5$, which are known by \S\S \ref{4},\ref{5}
produce, when appropriately completed, elements $x_\lambda\in L(\lambda)$
as required. This proves (1).

To see (2), observe that (1) shows that $N$ acts non-trivially
on the simple modules $L(\lambda)$ for $\lambda\in \Lambda^1$,
and so the set of $\lambda$ such that $N$ acts trivially
in $L(\lambda)$ is contained in $\Lambda^0$. But $|\Lambda^0|=r+1$,
and it is easy to see that $V^{\otimes r}$ has $r+1$ distinct
simple components (as $\fsl_2$-module). It follows that
$N$ acts trivially in at least $r+1$ of the simple modules
$L(\lambda)$, and (2) is immediate (cf. \ref{kereta}), as
is (3).

Since $\Phi\in\CP+\CR$, the latter is a two-sided ideal of $B$
which acts non-trivially on $L(\lambda)$ for $\lambda\in\Lambda^1$.
But $\Phi\in N$, so that $\CP+\CR$ acts trivially on $L(\lambda)$
for $\lambda\in\Lambda^0$. The statement (4) follows.
\end{proof}

Combined with the results of \S\ref{radical} Theorem \ref{lambda0} leads
to the following criterion for the truth of Conjecture \ref{conj}
\begin{cor}\label{radcriterion}
The conjecture \ref{conj} is equivalent to the following statement.
For each $\lambda\in\Lambda^0$
(as above) the $B_r$-submodule of $W(\lambda)$ generated by
$\Phi W(\lambda)$ contains $R(\lambda)$.
\end{cor}
\begin{proof}
It follows from Theorem \ref{lambda0}(4) that the conjecture is
equivalent to the statement that $\CP\supseteq \CR$. By
Theorem \ref{rad}(3) this is equivalent to the stated criterion.
\end{proof}

Next we show that the Conjecture is true for $r=5$.

\begin{prop}\label{prooffr5}
If $r=5$, $\CP=\langle\Phi\rangle$ contains the radical $\CR$ of $B$.
Hence by \ref{radcriterion} the conjecture is true for $r=5$.
\end{prop}
\begin{proof}
In view of Theorem \ref{r5}, it suffices to show that
$\CP W(\lambda)=R(\lambda)$ for $\lambda=(1^3)$ and $\lambda=(2,1)$.
But again by Theorem \ref{r5}, we have the situation of Corollary
\ref{corirred} here, whence it suffices to show simply that $\Phi$ acts
non-trivially on the cell modules $W(1^3)$ and $W(2,1)$. This will require
two calculations, which we now proceed to outline.

\noindent{\it The case $W(1^3)$.} In this case $M(\lambda)=\{((ij),\tau)\}$
where $(ij)$ is a transposition in $\Sym_5$ and $\tau$ is the unique standard
tableau of shape $1^3$. Thus we may write a basis for $W(1^3)$
as $\{C_{ij}\mid 1\leq i<j\leq 5\}$. Recalling that the Kazhdan-Lusztig
basis element of $\C\Sym_3$ corresponding to $(\tau,\tau)$ is
$E(3):=\sum_{w\in\Sym_3}\varepsilon(w)w$, the following facts are easily
verified using the diagrammatic representation of $B_5$.
$$
\begin{aligned}
s_1C_{45}=-C_{45};\;\;s_3C_{45}=&C_{35};\;\;FC_{45}=2(C_{45}-C_{35});\\
e_2C_{45}=0;\;\;e_2C_{35}&=-C_{23};\;\;e_{14}C_{23}=0.\\
\end{aligned}
$$
Using these equations one calculates in straightforward fashion that
$$
\Phi C_{45}=2(C_{23}- C_{13}- C_{24}+ C_{14}- C_{45}+ C_{35})\neq 0.
$$

\noindent{\it The case $W(2,1)$.} In this case $M(\lambda)=\{((ij),\tau_k)\}$
where $(ij)$ is a transposition in $\Sym_5$ and $\tau_k$ is one of the
two standard
tableau of shape $(2,1)$.
Explicitly,
$$
\tau_1=\begin{matrix}
1 & 3\\2 & \\
\end{matrix},\;\;\;\;\;
\tau_2=\begin{matrix}
1 & 2\\3 & \\
\end{matrix}
$$
Thus we may write a basis for $W(2,1)$
as $\{C_{ij,\tau_k}\mid 1\leq i<j\leq 5,\;\;k=1,2\}$.
In this case we need to recall that the Kazhdan-Lusztig cell
representation of $\C\Sym_3$ which corresponds to $(2,1)$ may be thought of
as having basis $\{c_{\tau_1},c_{\tau_2}\}$ and action by $\Sym_3=\langle r_1,r_2\rangle$
given by
$$
r_1c_{\tau_1}=-c_{\tau_1};\;\;r_1c_{\tau_2}=c_{\tau_2}-c_{\tau_1};\;\;
r_2c_{\tau_1}=c_{\tau_1}-c_{\tau_2};\;\;r_2c_{\tau_2}=-c_{\tau_2}.
$$
With these facts one verifies easily the following facts
$$
s_1C_{45,\tau_1}=-C_{45,\tau_1};\;\;s_3C_{45,\tau_1}=C_{35,\tau_1};\;\;
FC_{45,\tau_1}=2(C_{45,\tau_1}-C_{35,\tau_1}).
$$
Further,
$$
\begin{aligned}
e_2C_{45,\tau_1}=&0;\;\;e_2C_{35,\tau_1}=C_{23,\tau_1}-C_{23,\tau_2};\\
e_2FC_{45,\tau_1}=&2(C_{23,\tau_2}-C_{23,\tau_1});\;\;e_{14}C_{23,\tau_k}=0
\text{ for }k=1,2.\\
\end{aligned}
$$

Using these equations, it is straightforward to calculate that
\begin{equation}\label{case(21)}
\begin{aligned}
\Phi C_{45,\tau_1}=&2(C_{23,\tau_2}-C_{13,\tau_2}-C_{24,\tau_2}+C_{14,\tau_2}\\
&-C_{23,\tau_1}+C_{13,\tau_1}+C_{24,\tau_1}-C_{14,\tau_1}
-C_{45,\tau_1}+C_{35,\tau_1})\neq 0.
\end{aligned}
\end{equation}
This completes the proof of the Proposition.
\end{proof}

%

A computer calculation has been done to verify the case $r=6$.

\begin{thm}\label{comp} Let $\eta$ be the surjective homomorphism
from $B_r:=B_r(3)$ to $E_r:=\End_{\fsl_2(\C)}V(2)^{\otimes r}$,
and let $\Phi\in B_r$ be the element defined above. Then for
$r\leq 6$ $\Phi$ generates the kernel of $\eta$.
\end{thm}
\begin{proof} We have proved the result for $r\leq 5$.
The case $r=6$ was checked by a computer calculation, which verified that
$\dim \langle \Phi\rangle$ is correct in that case. Since
we know that $\Phi\in\Ker(\eta),$ the result follows.
\end{proof}

We are grateful to Derek Holt for doing this computation for us
using the Magma computational algebra package, with an
implementation of noncommutative Gr\"obner basis due to Allan Steel.

\section{The quantum case}

In this section, we develop the theme of \S\ref{bmw-basic}
and consider the BMW algebra $BMW_r(q)$ over
$\CA_q$, and its specialisation $BMW_r(\CK)$.
The results of the last
section on the Brauer algebra all generalise to the present
case, and we deduce some new ones through the technique of
specialisation. One of the key observations is 
that $BMW_r(q)$ has the $\C$ algebra $BMW_r(1)\cong B_r(3)$
(cf. Lemma \ref{bmwspec}) as a specialisation.

\subsection{Specialisation and cell modules.}\label{speccell}
In analogy with the case of the Brauer algebra,
of which it is a deformation,
$BMW_r(q)$ has a cellular structure 
\cite[Theorem 3.11]{X} and is also
quasi hereditary \cite[Theorem 4.3]{X}. For each partition
$\lambda\in\Lambda(r)$, there is therefore
a cell module $W_q(\lambda)$ of
dimension $w_\lambda$ for $BMW_r(q)$. Each cell module has a
non-zero irreducible head $L_q(\lambda)$, and these
irreducibles form a complete set of representatives of the
isomorphism classes of simple $BMW_r(q)$-modules. Furthermore,
$BMW_r(q)$ is semisimple if and only if all the cell modules are
simple (see \cite[\S 3]{X}).

Recall that $BMW_r(q)$ is the $\CA_q$-algebra defined by the 
presentation (\ref{braidgiq}),
where $\CA_q$ is the localisation of $\C[q^{\pm 1}]$ at
the multiplicative subset $\CS$
generated by $[2]_q$, $[3]_q$ and $[3]_q-1$. 
By Lemma \ref{bmwspec}, $BMW_r(q)$ may be thought of as an {\em integral
form} of $BMW_r(\CK)$.

One may identify $BMW_r(q)$ with the $\CA_q$-algebra 
generated by $(r,r)$-tangle diagrams, which satisfy
the usual relations (cf.
e.g., \cite[Definition 2.5]{X}). For each
Brauer $r$-diagram $T$ \cite[\S 4]{GL96}, 
it is explained in \cite[p. 285]{X} how to 
construct an $(r, r)$-tangle diagram $T_q$ by lifting each intersection
in $T$ to an appropriate crossing. The tangle diagrams obtained this
way form a basis of $BMW_r(q)$, which we shall denote by $\CT_q$.

The cell modules $W_{q}(\lambda)$ of $BMW_r(q)$ are
parametrised by partitions $\lambda\in \Lambda(r)$. They may also
be described diagramatically, in a similar way to the cell modules 
of the Brauer algebra $B_r(3)$ (cf. \S\ref{cell-br}). We proceed 
to give this description.
Let $t\in\CT(r)$; that is, $0\leq t\leq r$ and $r-t\in 2\Z$. 
For a partition $\lambda$ of $t$, we take $M(\lambda)$ to be the 
set defined in \S\ref{cell-br} for the Brauer algebra, viz
$M(\lambda)$ is the set of pairs $(S,\tau)$ where $S$ is an involution 
in $\Sym_r$ with $|\lambda|=t$ fixed points, and $\tau$ is a 
standard tableau of shape $\lambda$. In analogy with \S\ref{cell-br},
if $(S,\tau)$ and $(S',\tau')$ are two elements of $M(\lambda)$,
we obtain the (cellular) basis element 
$C^\lambda_{(S,\tau),(S',\tau')}(q)$ of $BMW_r(q)$ by 
\be\label{cellbasis}
C^\lambda_{(S,\tau),(S',\tau')}(q)
=\sum_{w\in\Sym_{t}}
P_{w(\tau,\tau'),w}(q)[S,S',w]_q,
\ee
where $C_v=\sum_{w\in\Sym_t} P_{v,w}(q)T_w$ is the Kazhdan-Lusztig
basis element of the Hecke algebra
$H_t(q^2)$, $[S,S',w]_q$ is the element of the 
basis $\CT_q$ (i.e. dangle) corresponding to the Brauer diagram
$[S,S',w]$, and all other notation is as in \S\ref{cell-br}.
Note that $P_{v,w}(q)\in\Z[q^{\pm 1}]\subset \CA_q$, so that 
$C^\lambda_{(S,\tau),(S',\tau')}(q)\in BMW_r(q)$.

Now for each element $(S,\tau)\in M(\lambda)$, the arguments
leading to \cite[Cor. 3.13]{X} describe how to associate to 
$(S,\tau)$ an $(r,t)$ dangle which we denote by $(S,\tau)_q$.
These form an $\CA_q$-basis of $W_q(\lambda)$, with the action
of $BMW_r(q)$ given by concatenation, using the relations
in \cite[Def. 2.5]{X} and the action of the Hecke algebra
$H_t(q^2)$ on its cell modules (which have basis $\{\tau\}$).
The next statement is a general result about cellular algebras,
adapted to our situation.

\begin{prop}\label{cellmodspec}
Let $\phi:\CA_q\lr R$ be a homomorphism of commutative rings
with $1$, and denote by $BMW_r^\phi$ the specialisation
$R\otimes_\phi BMW_r(q)$. Then
\begin{enumerate}
\item There is a natural bijection between the $\CA_q$-basis
$\{(S,\tau)_q\}$ of $W_q(\lambda)$ and an $R$-basis
of the specialised cell module $W^\phi(\lambda)$.
\item If $a\in BMW_r(q)$, the matrix of $1\otimes a\in BMW_r^\phi$
with respect to the basis in (1) is obtained by applying $\phi$
to the entries of the matrix of $a$.
\item The Gram matrix of the canonical form on $W^\phi(\lambda)$
is obtained from that of $W_q(\lambda)$ by applying $\phi$ 
to the entries of the latter.
\item If $W^\phi(\lambda)$ is simple, so is $W_q(\lambda)$.
\item We have 
$\rank_{\CA_q}L_q(\lambda)\geq \rank_{R}L^\phi(\lambda)$,
where $L_q(\lambda)$ is the simple head of $W_q(\lambda)$, etc.
\item For any pair $\mu\geq \lambda\in\Lambda(r)$, the
multiplicity $[W_q(\lambda):L_q(\mu)]\leq 
[W^\phi(\lambda):L^\phi(\mu)]$.
\end{enumerate}
\end{prop}
\begin{proof}
The bijection of (1) arises from any $\CA_q$-basis $\{\beta\}$
of $W_q(\lambda)$, by taking $\beta\mapsto 1\otimes \beta$. 
Given this, the assertion (2) is clear, as is (3). If $\Delta(\lambda)$
is the determinant of the Gram matrix of $W_q(\lambda)$ 
(i.e. the discriminant), the discriminant 
$\Delta^\phi(\lambda)$ of $W^\phi(\lambda)$ is given by 
$\Delta^\phi(\lambda)=\phi(\Delta(\lambda))$. If $W^\phi(\lambda)$
is simple, then $\Delta^\phi(\lambda)\neq 0$, whence
$\Delta_q(\lambda)\neq 0$. This implies that
if $\phi_q$ is the inclusion of $\CA_q$ in $\CK$, 
then $W^{\phi_q}(\lambda)(=W_\CK(\lambda))$ is simple as 
$BMW_r(\CK)$ module. It follows that $W_q(\lambda)$ has no
non-trivial $BMW_r(q)$-submodules, whence (4). Finally, 
note that $\rank_{\CA_q}(L_q(\lambda))$ equals the rank of the 
Gram matrix of the form. Since this cannot increase on 
specialisation, (5) follows. To see (6), observe that any
composition series of $W_q(\lambda)$ specialises 
(under the functor $R\otimes_\phi-$) to a chain
of submodules of $W^\phi(\lambda)$. But by (3),
the specialisation of $L_q(\mu)$ has $L^\phi(\mu)$ as a subquotient,
from which (6) follows.
\end{proof}

\subsection{An element of the quantum kernel}\label{quantker}
We next consider some elements of $BMW_r(q)$ which will play
an important role in the remainder of this work, and which will be used 
to define the Temperley-Lieb analogue of the title. Let $f_i= -g_i
-(1-q^{-2})e_i +q^2$, and set
\begin{eqnarray} \label{defFq}
F_q= f_1 f_3.
\end{eqnarray}
We also define $e_{1 4}=g_3^{-1} g_1 e_2 g_1^{-1} g_3$ and $e_{1 2 3
4} = e_2 g_1 g_3^{-1} g_2 g_1^{-1} g_3$.

The next two results are quantum analogues of Proposition \ref{propsPhi}.

\begin{lem}\label{relatPhiq}
The following identities hold in $BMW_4(q)$ (and hence in
$BMW_r(q)$).
\begin{eqnarray}
&&f_i=\frac{(g_i-q^2)(g_i-q^{-4})}{q^{-2}+q^{-4}}\\ 
&&e_i f_i=0,\quad f_i^2 = (q^2+q^{-2})f_i, \quad i=1, 2, 3, \label{relatPhiq1}\\
&&e_2 F_q e_2 =\tilde{a}e_2 -d e_{1 2 3 4} + a e_2 e_{1 4},\label{relatPhiq2}\\
&&e_2 F_q e_2 e_{1 4}=e_{1 4} e_2 F_q e_2 =(q^2+q^{-2})^2 e_2 e_{1
4},
\label{relatPhiq3}\\
&&e_2 F_q e_{1 2 3 4} =  e_{1 2 3 4} F_q e_2 = -d e_2 + a e_{1 2 3
4} + q^{-4}\tilde{a}e_2 e_{1 4}, \label{relatPhiq4}
\end{eqnarray}
where
\[
a=1+(1-q^{-2})^2, \quad \tilde{a}= 1+(1-q^2)^2, \quad d=
(q-q^{-1})^2=q^2(a-1)=q^{-2}(\tilde a-1).
\]
\end{lem}
The first relation follows easily from the relations
(\ref{extrarel}), and the others are straightforward consequences.
Note that $f_i=(q^2+q^{-2})d_i$, where $d_i$ is the idempotent of
Lemma \ref{proj-gen}.
Alternatively, one may use the representation of elements of the
BMW algebra by tangle diagrams, and the multiplication by
composition of diagrams, to verify the above statements.

Define the following element of $BMW_4(q)$:
\begin{eqnarray}\label{defPhiq}
\Phi_q &=&aF_q e_2 F_q - b F_q - c F_qe_2e_{14}F_q +d F_q e_{1 2 3 4} F_q,
\end{eqnarray}
where
\begin{eqnarray*}
\begin{aligned}
b&=1+(1-q^{2})^2+(1-q^{-2})^2,\\
c&= \frac{1 +(2+q^{-2})(1-q^{-2})^2 + (1+q^2)
(1-q^{-2})^4}{([3]_q-1)^2}.
\end{aligned}
\end{eqnarray*}

\begin{prop}\label{propsPhiq}
The elements $F_q, \Phi_q$ have the following properties:
\begin{enumerate}
\item $F_q^2= (q^2+q^{-2})^2 F_q$. \label{PropsPhiq1}
\item $e_i\Phi_q=\Phi_q e_i=0$ for $i=1,2,3$. \label{PropsPhiq2}
\item $\Phi_q^2 = -(q^2+q^{-2})^2 (1+(1-q^{2})^2+(1-q^{-2})^2 )
\Phi_q$. \label{PropsPhiq3}
\item $\Phi_q$ acts as $0$ on $V_q^{\otimes 4}$. \label{PropsPhiq4}
\end{enumerate}
\end{prop}
\begin{proof}
Part (\ref{PropsPhiq1}) immediately follows from the second relation
in \eqref{relatPhiq1}.

The fact that $e_1 \Phi_q = e_3 \Phi_q =0$ follows from the first
relation of \eqref{relatPhiq1} in Lemma \ref{relatPhiq}. Now
\[
e_2\Phi_q =a e_2 F_q e_2 F_q -b e_2 F_q-c e_2 F_q e_2 e_{1 4} F_q +d
e_2 F_q e_{1 2 3 4} F_q.
\]
Using the relations \eqref{relatPhiq2}, \eqref{relatPhiq3} and
\eqref{relatPhiq4}, we readily obtain $e_2\Phi_q=0$.
It can be similarly
shown that $\Phi_q e_i=0$ for $i=1, 2, 3$. Thus  by part
(\ref{PropsPhiq2}), we see $\Phi_q F_q e_2 =0$, and therefore that
$\Phi_q^2 = - b \Phi_q F_q=-(q^2+q^{-2})^2 b \Phi_q$.

The proof of part (\ref{PropsPhiq4}) proceeds in much the same way
as in the classical case. Note that $\Phi_q(V_q^{\otimes 4})\subset
L(2)_q\otimes L(2)_q$. Thus by Lemma \ref{e2iso},
$\Phi_q(V_q^{\otimes 4})\cong e_2\Phi_q(V_q^{\otimes 4})$ as
$\cU_q(\fsl_2)$-modules. Since $e_2\Phi_q=0$ by part (2), the proof
is complete.
\end{proof}

\subsection{A regular form of quantum $\mathfrak{sl}_2$}\label{regularform}

In this subsection we consider the quantised universal enveloping
algebra of $\mathfrak{sl}_2$ over the ring
$\CA_q$ and its representations. By ``regular form'' we shall understand
an $\CA_q$-lattice in a $\CK$-representation of $\cU_q(\fsl_2)$.
Denote by $\UK$ the $\CA_q$-algebra generated by $e, f,
k^{\pm 1}$ and $h:=\frac{k-k^{-1}}{q-q^{-1}}$, subject to the usual
relations, and call it the {\em regular form} of
$\cU_q(\mathfrak{sl}_2)$. Recall (Definition \ref{spec}) that
we have homomorphisms $\phi_1$ and $\phi_q$ from
$\CA_q$ to $\C, \CK$ respectively; the resulting specialisation
$\C\otimes_{\phi_1}\UK$ at $\phi_1$ is isomorphic to the universal
enveloping algebra of $\cU(\mathfrak{sl}_2)$ of $\mathfrak{sl}_2$
with
\[1\otimes e \mapsto \begin{pmatrix}0&1\\
0&0\end{pmatrix}, \quad 1\otimes f \mapsto \begin{pmatrix}0&0\\
1&0\end{pmatrix}, \quad 1\otimes h  \mapsto \begin{pmatrix}1&0\\
0&-1\end{pmatrix}, \quad 1\otimes k \mapsto \begin{pmatrix}1&0\\
0&1\end{pmatrix}.\]

The $\CA_q$-span $V^{reg}_q(2)$ of the vectors $v_0, v_{\pm 1}$ (see
Section \ref{qnotation}) forms a $\UK$-module, which is an
$\CA_q$-lattice in
$V_q(2)$.  Denote the $r$-th tensor power of $V^{reg}_q(2)$ over
$\CA_q$ by $V^{reg}_q(2)^{\otimes r}$; this has an $\CA_q$-basis consisting
of the elements $v_{i_1, i_2, \dots, i_r}$. Then
\begin{eqnarray}
\begin{aligned}
&\CK\otimes_{\phi_q}V^{reg}_q(2)^{\otimes r}= V_q(2)^{\otimes r},
\quad
\text{as $\cU_q(\fsl_2)$-module}; \\
&\C\otimes_{\phi_1}V^{reg}_q(2)^{\otimes r}\cong V(2)^{\otimes r},
\quad \text{as $\mathfrak{sl}_2$-module}.
\end{aligned}
\end{eqnarray}

\begin{rem}\label{tensor-basis}
The vectors $1\otimes v_{i_1, i_2, \dots, i_r}$ form a basis for
$\C\otimes_{\CA_q}V^{reg}_q(2)^{\otimes r}$. It follows that
$1\otimes v\in \C\otimes_{\phi_1}V^{reg}_q(2)^{\otimes r}$ is zero
if and only if $v\in (q-1)V^{reg}_q(2)^{\otimes r}$.
\end{rem}

Denote by $E_q^{reg}(r)$ the $\CA_q$ algebra 
$\End_{\UK}(V^{reg}_q(2)^{\otimes r})$. 
Recall (Theorem \ref{surjbmwend}) that we have the surjection 
$\eta_q:BMW_r(\CK)\lr E_q(r)=\CK\otimes_{\phi_q}E_q^{reg}(r)$. The next 
result shows that $\eta_q$ preserves the $\CA_q$-structures.


\begin{lem} \label{Endq-reg} We have 
$\eta_q(BMW_r(q))\subseteq E_q^{reg}(r)$.
In particular, if $[3]_q\inv e_i, d_i$ and $c_i$
are the idempotents of Theorem \ref{surjbmwend}, then 
$\eta_q(g_i)$, $\eta_q(\frac{e_i}{[3]_q})$, $\eta_q(d_i)$ and
$\eta_q(c_i)$ belong to $E_q^{reg}(r)$ for all $i$.
\end{lem}
\begin{proof}
The formul\ae\; in Lemma \ref{e-action} show explicitly that 
$\eta_q(\frac{e_i}{[3]_q})\in E_q^{reg}(r)$. A similar computation
shows that $\eta_q(d_i)\in E_q^{reg}(r)$, as follows. 
Evidently it suffices to treat the case $i=1$.
Write $\eta_q(d_1)(v_{k,l}):=x_{k,l}$; clearly we need only show that 
$x_{k,l}\in V_q^{reg}(2)^{\otimes 2}$ for $k,l=0,\pm 1$. But one verifies 
easily that the following explicit {formul\ae} describe the action
of $d_i$. Write $u_{-1}=q^{-2}v_{-1,0}-v_{0,-1},\;u_0=
-q^{-2}v_{-1,1}+(1-q^{-2})v_{0,0}+q^{-2}v_{1,-1}$ and 
$u_1=q^{-2}v_{0,1}-v_{1,0}$, and observe that $u_i\in V^{reg}_q(2)^{\otimes 2}$
for $i=0,\pm 1$. Then $x_{1,1}=x_{-1,-1}=0$, $x_{0,1}=\frac{1}{q^{-2}+q^2}u_1$,
$x_{0,1}=-q^2x_{1,0}$, $x_{0,0}=\frac{1-q^{-2}}{q^{-2}+q^2}u_0$,
$x_{-1,1}=-x_{1,-1}=\frac{1}{q^{-2}+q^2}u_0$, $x_{0,-1}=\frac{1}{q^{-2}+q^2}u_{-1}$,
and $x_{-1,0}=-q^2x_{0,-1}$. 

This shows that $\eta_q(d_i)\in E_q^{reg}(r)$, and since 
$\eta_q(\frac{e_i}{[3]_q})+\eta_q(d_i)+\eta_q(c_i)=1$,
it follows that
$\eta_q(c_i)\in E_q^{reg}(r)$. But $\eta_q(g_i)=
q^{-4}\eta_q(\frac{e_i}{[3]_q})-q^{-2}\eta_q(d_i)+q^2\eta_q(c_i)$,
whence the result.
\end{proof}

As a $\cU_q(\mathfrak{sl}_2)$-module, 
$V_q(2)^{\otimes r}\cong \CK\otimes_{\phi_q}
V^{reg}_q(2)^{\otimes r}$ 
is the direct sum of isotypic
components $I_q(2l)$, where every
irreducible $\cU_q(\mathfrak{sl}_2)$-submodule of $I_q(2l)$ 
has highest weight $2l$. It 
follows from the $\cU_q(\mathfrak{sl}_2)$ case of Theorem 8.5 in
\cite{LZ} that $I_q(2l)$ is an irreducible 
$\cU_q(\mathfrak{sl}_2)\otimes_{\CK} BMW_r(\CK)$-submodule 
of $V_q(2)^{\otimes r}$.

\begin{lem}\label{UxBr-irrep}
\begin{enumerate}
\item $I^{reg}_q(2l):=I_q(2l)\cap V^{reg}_q(2)^{\otimes r}$ 
is a $BMW_r(q)\otimes_{\CA_q} \UK$-submodule of $V^{reg}_q(2)^{\otimes r}$.
\item The specialisation $I(2l):=\C\otimes_{\phi_1}I^{reg}_q(2l)$ of
$I^{reg}_q(2l)$ is isomorphic as a $\cU(\mathfrak{sl}_2)$-module to
the isotypic component of $V(2)^{\otimes r}$ with highest weight $2l$.
\item $I(2l)$ is an irreducible $B_r(3)\otimes\cU(\mathfrak{sl}_2)$
-submodule of $V(2)^{\otimes r}$. 
.
\end{enumerate}
\end{lem}
\begin{proof}
By Lemma \ref{Endq-reg}, $I^{reg}_q(2l)$ is stable under the action
of $BMW_r(q)$. Since it is evidently a $\UK$-module and the
$\UK$ action commutes with the action of $BMW_r(q)$, part (1) follows.

In view of Remark \ref{tensor-basis}, $I(2l)$ is a non-trivial
subspace of $\C\otimes_{\phi_1}V^{reg}_q(2)^{\otimes r}$. By part
(1), $I(2l)$ is isomorphic to some $\cU(\mathfrak{sl}_2)\otimes
B_r(3)$ submodule of $V(2)^{\otimes r}$ whose $\mathfrak{sl}_2$-
submodules all have 
highest weight $2l$. The $\mathfrak{sl}_2$ case of Theorem 3.13 in
\cite{LZ} states that the $\cU(\mathfrak{sl}_2)$ isotypical
component of $V^{\otimes r}$ with highest weight $2l$ is the unique
irreducible $\cU(\mathfrak{sl}_2)\otimes B_r(3)$ submodule with this
$\mathfrak{sl}_2$ highest weight. This implies both parts (2) and (3).
\end{proof}


As a $BMW_r(\CK)$-module, $I_q(2l)$ is the direct sum of ${\rm
dim}_{\CK}V_q(2l)$ copies of a single irreducible $BMW_r(q)$-module, 
which we refer to as
$L^{BMW}_q(2l)$. Similarly, $I(2l)$ is the direct sum of ${\rm
dim}_\C V(2l)$ copies of an irreducible $B_r(3)$-module $L^{Br}(2l)$.
Recall that both ${\rm dim}_\CK V_q(2l)$ and ${\rm
dim}_\C V(2l)$ are equal to $2l+1$.

\begin{lem}\label{equal-dimension}
With notation as above, the irreducible $BMW_r(q)$-module
$L^{BMW}_q(2l)$ has the same dimension as that of the irreducible
$B_r(3)$-module $L^{Br}(2l)$.
\end{lem}
\begin{proof} If $l\ne l'$, $I(2l)$ and $I(2l')$ intersect
trivially since they are isotypical components with different
highest weights. Thus $\sum_l {\rm dim}I_q(2l)=3^r =\sum_l {\rm
dim}I(2l)$. But the specialisation argument of 
Proposition \ref{speccell}(5) shows that
${\rm dim}_\CK I_q(2l) \ge {\rm dim}_\C I(2l)$, whence \[{\rm
dim}_\CK I_q(2l) = {\rm dim}_\C I(2l).\]  Thus
\[
{\rm dim}_\CK L^{BMW}_q(2l) = \frac{{\rm dim}_\CK I_q(2l)}{2l+1}= \frac{{\rm
dim}_\C I(2l)}{2l+1}={\rm dim}_\C L^{Br}(2l).\]
\end{proof}

Denote by $\CR(\CK)$ the radical of the BMW algebra $BMW_r(\CK)$ and let
$\overline{BMW}_r(\CK)=BMW_r(\CK)/\CR(\CK)$ be its largest semisimple
quotient. Then as explained in \S \ref{radical},
$\overline{BMW}_r(\CK) = \oplus_{\lambda\in\Lambda}
\overline{B}_\CK(\lambda)$ with $\overline{B}_\CK(\lambda) \cong
\rm{End}_{\CK}(L_\CK(\lambda))$, where $L_\CK(\lambda)$ is the simple
head of the cell module $W_\CK(\lambda)$. 
As in Lemma \ref{annihilators}, the surjective algebra
homomorphism $\eta_q: BMW_r(\CK)\to {\rm
End}_{\cU_q(\fsl_2)}(V_q^{\otimes r})$ induces a surjection
$\overline{\eta}_q: \overline{BMW}_r(\CK)\to {\rm
End}_{\cU_q(\fsl_2)}(V_q^{\otimes r})$.

Similarly, let $\overline{B}_r(3)$ denote the largest semi-simple
quotient of the Brauer algebra. Then
$\overline{B}_r(3)=\oplus_{\lambda\in\Lambda} \overline{B}(\lambda)$
with $\overline{B}(\lambda) \cong \rm{End}_{\C}(L(\lambda))$, where
$L(\lambda)$ is the simple head of the cell module $W(\lambda)$. Let
$\overline{\eta}: \overline{B}_r(3) \to {\rm
End}_{\fsl_2}(V^{\otimes r})$ be the surjection induced by the map
$\eta: B_r(3)\to {\rm End}_{\fsl_2}(V^{\otimes r})$.

Recall that in analogy with (\ref{add}), we have 
\be\label{Bqlambdageneric}
BMW_r(q)= \oplus_{\lambda\in\Lambda} B_{\CA_q}(\{\lambda\}), \quad
\text{where}\quad B_{\CA_q}(\{\lambda\})=\sum_{S, T\in M(\lambda)} \CA_q
C^\lambda_{S, T}.
\ee
Taking appropriate tensor products 
with $\CK$ and $\C$ respectively, and writing
$B_\C(\{\lambda\})$ for what was denoted $B(\{\lambda\})$ in \S\S 
\ref{radical},\ref{classical3}, we obtain

\be\label{Bqlambda}
\begin{aligned}
BMW_r(\CK)= &\oplus_{\lambda\in\Lambda} B_{\CK}(\{\lambda\}),
\text{ where } B_{\CK}(\{\lambda\})=\sum_{S, T\in M(\lambda)} \CK
C^\lambda_{S, T},\text{ and }\\
BMW_r(\C)= &\oplus_{\lambda\in\Lambda} B_{\C}(\{\lambda\}), 
\text{ where } B_{\C}(\{\lambda\})=\sum_{S, T\in M(\lambda)} \C
C^\lambda_{S, T}.\\
\end{aligned}
\ee

\begin{prop}\label{correspondence}
Maintain the above notation. Then
$\overline{\eta}_q(\overline{B}_q(\lambda))\ne 0$ if and only if
$\overline{\eta}(\overline{B}(\lambda))\ne 0$. For such $\lambda$,
we have
${\rm dim}_\CK\overline{B}_q(\lambda) = {\rm dim}_\C\overline{B}(\lambda)$.
\end{prop}
\begin{proof}
Let $I_q^{reg}(\lambda):= B_{\CA_q}(\{\lambda\})(V_q^{reg}(2)^{\otimes
r})$, where $B_{\CA_q}(\{\lambda\})$ is defined by equation
\eqref{Bqlambdageneric}. Set $I_q(\lambda):=\CK\otimes_{\phi_q}
I_q^{reg}(\lambda)$ and $I(\lambda):=\C\otimes_{\phi_1}
I_q^{reg}(\lambda)$. Then
\begin{eqnarray*}
I(\lambda)&=&B_\C(\{\lambda\})(\C\otimes_{\phi_1} V_q^{reg}(2)^{\otimes
r}).
\end{eqnarray*}
Note that $I_q(\lambda)$ is a $\cU_q(\fsl_2)$-isotypic component
of $V_q(2)^{\otimes r}$, and $I(\lambda)$ is isomorphic to an
$\fsl_2$-isotypic component of $V(2)^{\otimes r}$. The
$\cU_q(\fsl_2)$-highest weight of $I_q(\lambda)$ is equal to the
$\fsl_2$-highest weight of $I(\lambda)$.

If $\overline{\eta}_q(\overline{B}_q(\lambda))=0$, then
$I_q(\lambda)=0$. In this case, $I(\lambda)=0$ and this is
equivalent to $\overline{\eta}(\overline{B}(\lambda))=0$. If
$\overline{\eta}_q(\overline{B}_q(\lambda))\ne 0$, then
$I_q(\lambda)\ne 0$, and it follows from Remark \ref{tensor-basis}
that $I(\lambda)\ne 0$. Therefore
$\overline{\eta}(\overline{B}(\lambda))\ne 0$.

With the first statement of the Proposition established, the second
follows immediately from Lemma \ref{equal-dimension}.
\end{proof}

Recall that $\Lambda^0$ denotes the set of all partitions with $3$
or fewer boxes in the first two columns, and
$\Lambda^1=\Lambda(r)\setminus \Lambda_0$. We have the following
analogue of Theorem \ref{lambda0}.
\begin{thm}\label{lambdaq}
Let $N_\CK$ be the kernel of the surjective map $\eta_q: BMW_r(\CK)\to
{\rm End}_{\cU_q(\fsl_2)}(V_q(2)^{\otimes r})$. Denote by $\CP_\CK$ the
two-sided ideal of $BMW_r(\CK)$ generated by $\Phi_q$, and by $\CR_\CK$
the radical of $BMW_r(\CK)$.
\begin{enumerate}
\item $\rm{End}_{\cU_q(\fsl_2)}(V_q(2)^{\otimes
r})\cong\oplus_{\lambda\in\Lambda^0}{\ov B}_q(\lambda)$.
\item If $\lambda\in\Lambda^1$, then $\Phi_q (L_q(\lambda)) \ne 0$.
\item $N_\CK $ acts trivially on $L_q(\lambda)$ if and only if
$\lambda\in\Lambda^0$.
\item  $\CP_\CK+\CR_\CK=N_\CK$.
\end{enumerate}
\end{thm}
\begin{proof}
Part (1) is an easy corollary of Proposition \ref{correspondence} in
view of Theorem \ref{lambda0}(3).

For any $\lambda\in\Lambda^1$, $B_r(3)\Phi W(\lambda) = W(\lambda)$
by Theorem \ref{lambda0}(1) and the cyclic property of $W(\lambda)$.
Since $B_r(3)\Phi W(\lambda)\cong \C\otimes_{\phi_1}
BMW_r(\CK)\Phi_q W_{\CA_q}(\lambda)$ and $W(\lambda)\cong
\C\otimes_{\phi_1}W_{\CA_q}(\lambda)$, it follows that $BMW_r(\CK)\Phi_q
W_{\CA_q}(\lambda)=W_{\CA_q}(\lambda)$ since
$\CK\otimes_{\CA_q}W_{\CA_q}(\lambda)$ and $W(\lambda)$ have the
same dimensions. This implies part (2).

The proof of parts (3) and (4) is essentially the same as that of
Theorem \ref{lambda0}(2), (4), and will  be omitted.
\end{proof}

\begin{rem}\label{lambdageneric}
\begin{enumerate}
\item
Although Theorem \ref{lambdaq} has been stated over $\CK$, it is clear that 
the statements (1)-(4) hold integrally, 
i.e. if we replace $\CK$ by $\CA_q$ and all
$\CK$ vector spaces by the corresponding free $\CA_q$-modules.
\item Note that (2) and (3) of Theorem \ref{lambdaq} imply that 
$\Phi_q(L_q(\lambda))=0$ if and only if $\lambda\in\Lambda^0$. 
This is because $\Phi_q\in N_q$ implies (by (3)) that $\Phi_q$
acts trivially on $L_q(\lambda)$ for $\lambda\in\Lambda^0$, while
(3) shows that $\Phi_q(L_q(\lambda))\neq 0$ for $\lambda\in\Lambda
\setminus\Lambda^0$.
\end{enumerate}
\end{rem}

The final result of this section is that to determine whether $\Phi_q$
generates $N_q$, it suffices to check the classical case.

\begin{prop}\label{classquant}
With notation as in Theorem \ref{lambdaq}, if $\langle\Phi\rangle=\CP$
contains $\CR$, the radical of $B_r(3)$, then $\langle\Phi_q\rangle_{BMW_r(\CK)}
=\CP_\CK$
contains the radical $\CR_\CK$ of $BMW_r(\CK)$.
\end{prop}
\begin{proof}
We have already noted that by \cite[Theorem 3.11]{X}, $BMW_r(\CK)$ is
cellular, with the canonical anti-involution being defined by 
$g_i^*=g_i$ and $e_i^*=e_i$. It follows that $\Phi_q^*=\Phi_q$,
and hence that $\CP_\CK$ is a self-dual two sided ideal of $BMW_r(\CK)$.
Hence we may apply Theorem \ref{rad} to deduce that $\CP_\CK\supseteq\CR_\CK$
if and only if $\CP_\CK W_\CK(\lambda)=R_\CK(\lambda)$ for each 
$\lambda\in\Lambda^0$, where $\Lambda^0$ is as in Theorem \ref{lambdaq}.

Now we are given that $\CP\supseteq\CR$, whence $\CP W(\lambda)=R(\lambda)$,
where $W(\lambda)=W_\C(\lambda)$, etc. Write 
$\CP_q:=\langle\Phi_q\rangle_{BMW_r(q)}$.

Let $\lambda\in\Lambda^0$ and consider the $\CA_q$-submodule 
$\CP_qW_{\CA_q}(\lambda)$ of $R_{\CA_q}(\lambda)$. For any element
$r\in R_{\CA_q}(\lambda)$, since $1\otimes_{\phi_1} r\in
1\otimes_{\phi_1}\CP_qW_{\CA_q}(\lambda)$, there exist elements
$r_0\in \CP_qW_{\CA_q}(\lambda)$ and 
$w=(q-1)r_1\in (q-1)W_{\CA_q}(\lambda)$
such that $r=r_0+w$. But $\CP_qW_{\CA_q}(\lambda)\subseteq 
R_{\CA_q}(\lambda)$ by Theorem \ref{lambdaq}(3), whence $w\in 
R_{\CA_q}(\lambda)$, and so evidently $r_1\in R_{\CA_q}(\lambda)$,
since $r_1$ has zero inner product with $W_{\CA_q}(\lambda)$.


It follows that multiplication by $q-1$ is an invertible
endomorphism of the quotient $R_q(\lambda)/\CP_qW_{\CA_q}(\lambda)$,
whence the latter is an $\CA_q$-torsion module. 
Hence $\CK\otimes_{\phi_q}(R_q(\lambda)/\CP_qW_{\CA_q}(\lambda))=0$, 
i.e. $R_\CK(\lambda)=\CP_\CK W_{\CK}(\lambda)$.
\end{proof}

\begin{cor}\label{1impliesq}\begin{enumerate}
\item
With the above notation, if $\Phi$ generates $N$ 
then $\Phi_q$ generates $N_\CK$.
\item If $r\leq 6$, then $\Phi_q$ generates $N_\CK$
as an ideal of $BMW_r(\CK)$.
\end{enumerate}
\end{cor}

The first statement is evident from Proposition \ref{classquant},
while the second follows from the first, together with Theorem
\ref{comp}.

We end this section by noting that our results imply the quantum 
analogues of the results of \S 6.

\begin{cor}\label{qu45}
\begin{enumerate}
\item The algebra $BMW_4(\CK)$ is semisimple.
\item The cell modules of $BMW_5(\CK)$ are all simple except for
those corresponding to the partitions $(2, 1)$ and $(1^3)$, whose
simple heads have dimensions $15, 6$ respectively.
\item The cell modules $W_\CK(1^3)$ and $W_\CK(2,1)$ have two composition
factors each.
\item The radical of $BMW_5(\CK)$ has dimension $239$.
\end{enumerate}
\end{cor}
\begin{proof}
All statements are easy consequences of Proposition \ref{cellmodspec}.
For example, it follows from {\it loc. cit.}(5) and (6), that if 
$W^\phi(\lambda)$ has just two composition factors, then either
$W_q(\lambda)$ is irreducible, or it also has two composition factors,
whose dimensions are the same as those of $W^\phi (\lambda)$. This
implies the statements (3), (4) and (5) above.
\end{proof}

\section{A BMW-analogue of the Temperley-Lieb algebra}\label{analogue}

Although implicit above, we complete this work with an explicit 
definition of our analogue of the Temperley-Lieb algebra, together with 
some of its properties, as well as some questions about it.

\begin{definition}\label{tlanalogue}
Let $\CA_q$ be the ring $\C[q^{\pm 1}, [3]_q\inv, (q+q\inv)\inv, (q^2+q^{-2})\inv]$.
The $\CA_q$-algebra $P_r(q)$ has generators $\{g_i^{\pm 1},e_i\mid 1=1,\dots,r-1\}$
and relations given by (\ref{braidgiq}) together with 
$\Phi_q=0$, where $\Phi_q$ is the word in the generators defined
in (\ref{defPhiq}). We reproduce the relations here for convenience.

\begin{equation*}
\begin{aligned}
g_ig_j&=g_jg_i\text{ if }|i-j|\geq 2\\
g_ig_{i+1}g_i&=g_{i+1}g_ig_{i+1} \text{ for }1\leq i\leq r-1\\
g_i-g_i\inv&=(q^2-q^{-2})(1-e_i)\text { for all }i\\
g_ie_i&=e_ig_i=q^{-4}e_i\\
e_ig_{i-1}^{\pm 1}e_i&=q^{\pm 4}e_i\\
e_ig_{i+1}^{\pm 1}e_i&=q^{\pm 4}e_i\\
\end{aligned}
\end{equation*}
\begin{eqnarray*}
\Phi_q &=&aF_q e_2 F_q - b F_q - c F_qe_2e_{14}F_q +d F_q e_{1 2 3 4} F_q=0,
\end{eqnarray*}
where
\begin{eqnarray*}
\begin{aligned}
f_i&= -g_i-(1-q^{-2})e_i +q^2,\\
F_q&=f_1f_3,\\
e_{1 4}&=g_3^{-1} g_1 e_2 g_1^{-1} g_3,\\
e_{1 2 3 4}&= e_2 g_1 g_3^{-1} g_2 g_1^{-1} g_3,\\
a&=1+(1-q^{-2})^2,\;\tilde{a}= 1+(1-q^2)^2,\\
d&=(q-q^{-1})^2=q^2(a-1)=q^{-2}(\tilde a-1),\\
b&=1+(1-q^{2})^2+(1-q^{-2})^2,\\
c&= \frac{1 +(2+q^{-2})(1-q^{-2})^2 + (1+q^2)
(1-q^{-2})^4}{([3]_q-1)^2}.
\end{aligned}
\end{eqnarray*}

\end{definition}

\subsection{Properties of $P_r(q)$}

Let $\phi_q:\CA_q\hookrightarrow \CK(=\C(q^{\frac{1}{2}}))$ be the 
inclusion map, and let $\phi_1:\CA_q\lr\C$ be defined by $\phi_1(q)=1$.
Write $P_r(\CK):=\CK\otimes_{\phi_q}P_r(q)$, and 
$P_r(\C):=\C\otimes_{\phi_1}P_r(q)$.
Then there are surjective homomorphisms
\be\label{etaq}
\begin{aligned}
\eta_q:P_r(\CK)&\lr \End_{\cU_q(\fsl_2)} V_q(2)^{\otimes r},\text{ and }\\
\eta:P_r(\C)&\lr \End_{\fsl_2(\C)} V(2)^{\otimes r},\\
\end{aligned}
\ee
where $V(2)$ is the two-dimensional irreducible $\fsl_2(\C)$-module
and $V_q(2)$ is its quantum analogue.

Moreover, it follows from Theorem \ref{lambdaq}(4)
and that $\Ker(\eta_q)$ is the radical of $P_r(\CK)$,
i.e., that $\End_{\cU_q(\fsl_2)} V_q(2)^{\otimes r}$ is the largest semisimple 
quotient of $P_r(\CK)$. A similar statement applies to $P_r(\C)$.

\subsection{Some open problems} We finish with 
some problems relating to $P_r(q)$.

\smallskip
\begin{enumerate}
\item Determine whether $P_r(q)$ is generically semisimple, in particular
whether $P_r(\CK)$ is semisimple. By Proposition \ref{classquant}, this is 
true provided that $P_r(\C)$ is semisimple. The latter algebra has been shown 
(Proposition \ref{prooffr5}) to be semisimple for $r\leq 5$ and the case $r=6$
has been verified by computer.

\item A question equivalent to (1) is to determine whether $P_r(\CK)$ has
dimension given by the formula (\ref{d2r}). More explicitly,
we know that 
\be\label{dimin}\dim_\CK P_r(\CK)\geq 
\binom{2r}{r}+\sum_{p=0}^{r-1}\binom{2r}{2p}\binom{2p}{p}
\frac{3p-2r+1}{p+1},
\ee
with equality if and only if 
the ideal of $BMW_r(\CK)$ which is generated by $\Phi_q$ contains the radical
$\CR(\CK)$ of $BMW_r(\CK)$. 

We therefore ask whether equality holds in (\ref{dimin}).

\item Is $P_r(q)$ free as $\CA_q$-module?

\item Determine whether $P_r(q)$ has a natural  cellular structure.

\item Generalise the program of this work to higher dimensional
representations of quantum $\fsl_2$.

\end{enumerate}

Finally, we note that an affirmative answer to Conjecture \ref{conj}
implies an affirmative answer to both (1) and (2) above.

\end{document}